\documentclass[11pt,a4paper]{amsart}

\usepackage{hyperref}
\usepackage{graphicx}
\usepackage{color}
\usepackage{url}
\usepackage{tikz}
\usepackage{cleveref}
\Crefname{thm}{Theorem}{Theorems}
\Crefname{conj}{Conjecture}{Conjectures}
\Crefname{que}{Question}{Questions}

\usepackage{pinlabel}
\usetikzlibrary{calc}
\usetikzlibrary{shapes.geometric}

\newtheorem{thm}{Theorem}
\numberwithin{thm}{section}
\numberwithin{equation}{section}
\newtheorem{cor}[thm]{Corollary}
\newtheorem{que}[thm]{Question}
\newtheorem{proposition}[thm]{Proposition}
\newtheorem{algorithm}[thm]{Algorithm}
\newtheorem{lemma}[thm]{Lemma}
\newtheorem{conj}[thm]{Conjecture}

\newcommand{\R}{\mathbf{R}}

\newcommand{\Z}{\mathbf{Z}}
\newcommand{\calF}{\mathcal{F}}
\newcommand{\Mod}{\operatorname{Mod}}

\begin{document}
\title{Minimal pseudo-Anosov stretch factors on nonoriented surfaces}
\author{Livio Liechti}
\address{Department of Mathematics, University of Fribourg, Ch.\ du Mus\'ee 23, \newline1700 Fribourg, Switzerland}
\email{livio.liechti@unifr.ch}
\author{Bal\'azs Strenner}
\address{Georgia Institute of Technology, School of Mathematics, Atlanta GA 30332, USA}
\email{strenner@math.gatech.edu}

\begin{abstract}
  We determine the smallest stretch factor among
  pseudo-Anosov maps with an orientable invariant foliation on the closed
  nonorientable surfaces of genus 4, 5, 6, 7, 8, 10, 12, 14, 16, 18 and 20. We
  also determine the smallest stretch factor of an orientation-reversing
  pseudo-Anosov map with orientable invariant foliations on
  the closed orientable surfaces of genus 1, 3, 5, 7, 9 and~11. As a byproduct,
  we obtain that the stretch factor of a pseudo-Anosov map on a nonorientable
  surface or an orientation-reversing pseudo-Anosov map on an orientable
  surface does not have Galois conjugates on the unit circle. This shows that
  the techniques that were used to disprove Penner's conjecture on orientable
  surfaces are ineffective in the nonorientable cases.
\end{abstract}
\maketitle

\section{Introduction}
\label{sec:intro}

Let $S$ be a surface of finite type. A homeomorphism $f$ of $S$ is
\emph{pseudo-Anosov} if there are transverse singular measured foliations
$\calF^u$ and $\calF^s$ and a real number $\lambda > 1$ such that
$f(\calF^u) = \lambda \calF^u$ and
$f(\calF^s) = \lambda^{-1} \calF^s$ \cite{Thurston88}. The
number $\lambda$ is called the \emph{stretch factor} of $f$.

Denote by $N_g$ the closed nonorientable surface of genus $g$ (the connected
sum of $g$ projective planes) and by $\delta^+(N_g)$ the minimal stretch factor
among pseudo-Anosov homeomorphisms of $N_g$ with an orientable invariant
foliation. (Only one of the foliations can be orientable, otherwise the surface
would have to be orientable as well.) The number $\delta^+(N_g)$ exists,
because for any surface $S$, the set of pseudo-Anosov stretch factors on $S$ is
a discrete set \cite{ArnouxYoccoz81, Ivanov88}.

\begin{thm}\label{theorem:stretch_factors_nonor}
  The values and minimal polynomials of $\delta^+(N_g)$ for $g=4$, $5$, $6$,
  $7$, $8$, $10$, $12$, $14$, $16$, $18$ and $20$ are as follows:
  \begin{center}
    \begin{tabular}{c|c|c|c}
      $g$ & $\delta^+(N_g)\approx$ & Minimal polynomial of
      $\delta^+(N_g)$ & singularity type\\
      \hline
      4 & 1.83929 & $x^3-x^2-x-1$ & (6) \\
      5 & 1.51288 & $x^4-x^3-x^2+x-1$ & (4,4,4) \\
      6 & 1.42911 & $x^5-x^3-x^2-1$ & (10) \\
      7 & 1.42198 & $x^6-x^5-x^3+x-1$ & (4,4,4,4,4) \\
      8 & 1.28845 & $x^7-x^4-x^3-1$ & (14) \\
      10 & 1.21728 & $x^9-x^5-x^4-1$ & (18) \\
      12 & 1.17429 & $x^{11}-x^6-x^5-1$ & (22) \\
      14 & 1.14551 & $x^{13}-x^7-x^6-1$ & (26) \\
      16 & 1.12488 & $x^{15}-x^8-x^7-1$ & (30) \\
      18 & 1.10938 & $x^{17}-x^9-x^8-1$ & (34) \\
      20 & 1.09730 & $x^{19}-x^{10}-x^9-1$ & (38) \\
    \end{tabular}
  \end{center}
  The table also contains the singularity type of the minimizing pseudo-Anosov
  map. For example, (4,4,4) means that the pseudo-Anosov map has three
  4-pronged singularities.
\end{thm}

Based on this result, we conjecture the following.

\begin{conj}\label{conj:nonor-polynomials}
  For all $k \ge 2$, $\delta^+(N_{2k})$ is the largest root of
  \begin{displaymath}
    x^{2k-1}-x^k-x^{k-1}-1.
  \end{displaymath}
\end{conj}

We think that the minimal stretch factors in the genus 9 and 11 are as follows.
We will discuss supporting evidence in \Cref{sec:final-section}.

\begin{conj}\label{conj:nine_eleven}
  The approximate values and minimal polynomials of $\delta^+(N_g)$ for $g=9,11$ are as follows:
  \begin{center}
    \begin{tabular}{c|c|c|c}
      $g$ & $\delta^+(N_g)\approx$ & Minimal polynomial of
      $\delta^+(N_g)$ & singularity type\\
      \hline
      9 & 1.35680 & $x^8 - x^5 - x^4 - x^3 - 1$ & (16) \\
      11 & 1.22262 & $\frac{x^{12} - x^{7} - x^{6} - x^{5} - 1}{x^{2}
                              + x + 1}$ & (8,8,8)\\
    \end{tabular}
  \end{center}
\end{conj}

For our second main result, denote by $\delta^+_{rev}(S_g)$ the minimal stretch
factor among orientation-reversing pseudo-Anosov homeomorphisms of the closed
orientable surface~$S_g$ of genus $g$ that have orientable invariant
foliations.

\begin{thm}\label{theorem:stretch_factors_rev}
  The values and minimal polynomials of $\delta^+_{rev}(S_g)$ for
  $g=1$, $3$, $5$, $7$, $9$ and $11$ are as follows:
  \begin{center}
    \begin{tabular}{c|c|c|c}
      $g$ & $\delta^+_{rev}(S_g)\approx$ & Minimal polynomial of
      $\delta^+_{rev}(S_g)$ & singularity type\\
      \hline
      1 & 1.61803 & $x^2-x-1$ & no singularities \\
      3 & 1.25207 & $\frac{x^8-x^5-x^3-1}{x^2+1}$ & (4,4,4,4) \\
      5 & 1.15973 & $\frac{x^{12}-x^7-x^5-1}{x^2+1}$ & (6,6,6,6) \\
      7 & 1.11707 & $\frac{x^{16}-x^9-x^7-1}{x^2+1}$ & (8,8,8,8) \\
      9 & 1.09244 & $\frac{x^{20}-x^{11}-x^9-1}{x^2+1}$ & (10,10,10,10) \\
      11 & 1.07638 & $\frac{x^{24} - x^{13} - x^{11} - 1}{x^2+1}$ & (12,12,12,12) \\
    \end{tabular}
  \end{center}
  Moreover, we have
  \begin{displaymath}
    \delta^+_{rev}(S_g) \ge \delta^+_{rev}(S_{g-1})
  \end{displaymath}
  for $g=2, 4, 6, 8$ and $10$.
\end{thm}

Based on these results, it is natural to conjecture the following.

\begin{conj}\label{conj:rev-polynomials}
  For all $k \ge 2$, $\delta^+_{rev}(S_{2k-1})$ is the largest root of
  \begin{displaymath}
    x^{4k}-x^{2k+1}-x^{2k-1}-1.
  \end{displaymath}
\end{conj}

\Cref{theorem:stretch_factors_rev} shows that $\delta^+_{rev}(S_g)$ fails to be
strictly decreasing at every other step for small values of $g$. We conjecture
that in fact the value of $\delta^+_{rev}(S_g)$ strictly increases at every
other step. We will discuss evidence for this after
\Cref{prop:elimination-reversing}.

\begin{conj}\label{conj:zigzag}
  For all $k\ge 1$, we have
  \begin{displaymath}
    \delta^+_{rev}(S_{2k}) > \delta^+_{rev}(S_{2k-1}).
  \end{displaymath}
\end{conj}

\subsection*{Motivation}
\label{sec:motivation}

One motivation for studying $\delta(S)$, the smallest stretch factor of an
orientation-preserving pseudo-Anosov map on an orientable surface, is that the
shortest closed geodesic (in the Teichm\"uller metric) on the moduli space of
algebraic curves homeomorphic to $S$ has length $\log\delta(S)$.

Another motivation for studying small stretch factors comes from 3-manifold
theory. The mapping torus of a pseudo-Anosov map $f$ is a hyperbolic 3-manifold
$M_f$ and the stretch factor of $f$ is related to the hyperbolic volume of
$M_f$ (\cite{KinKojimaTakasawa09,KojimaMcShane18}). This relates small-volume
hyperbolic manifolds (\cite{Agol02, AgolStormThurston07,
  GabaiMeyerhoffMilley09, Milley09}) to small stretch factor pseudo-Anosov
maps.

One reason to study minimal stretch factors specifically in nonorientable
settings is the following connection to the conjecture of Schinzel and Zassenhaus
that asserts the existence of a universal constant~$c>0$
such that for any algebraic integer that is not a root of unity,
the largest modulus among its Galois conjugates is bounded from below by
$1+c/d,$
where~$d$ is the degree of the algebraic integer~\cite{SchinzelZassenhaus65}.
Using a result due to Breusch~\cite{Breusch51}, the first author has shown that this conjecture
has an equivalent reformulation that
compares, for each genus, the minimal spectral radius~$>1$ among homological
actions of orientation-preserving mapping classes with the minimal spectral radius~$>1$
among homological actions of orientation-reversing mapping classes~\cite{LiechtiSZ}.
For instance, if for all but finitely many genera, one can obtain smaller spectral radii by
orientation-reversing mapping classes, the conjecture of Schinzel and Zassenhaus
is true~\cite{LiechtiSZ}.

By studying the stretch factors of pseudo-Anosov mapping classes with an orientable
invariant foliation, we restrict to a certain class of homological actions; however, for
this class, our results seem to suggest that indeed smaller spectral radii can typically
be obtained by orientation-reversing mapping classes, compare \Cref{conj:rev-limit}
with \Cref{eq:or-limit} below.

\subsection*{Previous results}
\label{sec:history}

The value of $\delta(S_g)$ for hyperbolic $S_g$ is only known for
$g = 2$ \cite{ChoHam08}. This value is the largest root
of $x^4-x^3-x^2-x+1$, which is approximately 1.72208.
More is known about $\delta^+(S_g)$, the minimal stretch factor of
orientation-preserving pseudo-Anosov maps on $S_g$ with orientable invariant
foliations. The known values are summarized in \Cref{tab:classical-results}
below.
\renewcommand{\arraystretch}{1.2}
\begin{table}[ht]
  \centering
  \begin{tabular}{c|c|c}
    $g$ & $\delta^+(S_g)\approx$ & Minimal polynomial of
                                   $\delta^+(S_g)$ \\
    \hline
    1 & 2.61803 & $x^2-3x+1$ \\
    2 & 1.72208 & $x^4-x^3-x^2-x+1$ \\
    3 & 1.40127 & $x^6-x^4-x^3-x^2+1$ \\
    4 & 1.28064 & $x^8-x^5-x^4-x^3+1$ \\
    5 & 1.17628 & $x^{10}+x^9-x^7-x^6-x^5-x^4-x^3+x+1 = \frac{x^{12}-x^7-x^6-x^5+1}{x^2 - x + 1}$ \\
    7 & 1.11548 & $x^{14} +x^{13}-x^9-x^8-x^7-x^6-x^5+x+1$ \\
    8 & 1.12876 & $x^{16}-x^9-x^8-x^7+1$ \\
  \end{tabular}
  \caption{The known values of $\delta^+(S_g)$.}
  \label{tab:classical-results}
\end{table}

Initially, the pseudo-Anosov maps realizing the stretch factors in \Cref{tab:classical-results}
were constructed in different ways. The construction is due to Zhirov
\cite{Zhirov95} for $g=2$, Lanneau \& Thiffeault \cite{LanneauThiffeault11} for
$g=3,4$, Leininger \cite{Leininger04} for $g=5$, Kin \& Takasawa
\cite{KinTakasawa13} and Aaber \& Dunfield \cite{AaberDunfield10} for $g=7$ and
Hironaka \cite{Hironaka10} for $g=8$. Hironaka \cite{Hironaka10} then showed
that all of the examples above except the $g=7$ example arise from the
fibration of a single hyperbolic 3-manifold, the mapping torus of the
``simplest hyperbolic braid''. The fact that the values in \Cref{tab:classical-results} are indeed
the minimal stretch factors was shown by Lanneau and Thiffeault
\cite{LanneauThiffeault11} by a systematic way of narrowing down the set of
possible minimal polynomials of the minimal stretch factors. For the larger
values of $g$, their proof is computer-assisted.

\subsection*{Asymptotics}
\label{sec:asymptotics}

The rough asymptotic behavior of $\delta(S)$ is well-understood: Penner
\cite{Penner91} showed that $\log\delta(S_g) \sim \frac1g$. For other
constructions of small stretch factors and asymptotics for different sequences
of surfaces, see
\cite{Bauer92,McMullen00,Minakawa06,HironakaKin06,Tsai09,Valdivia12,Yazdi18}.

Since the larger root of the polynomial $x^2-2x-1$ is $1+\sqrt{2}$,
\Cref{conj:nonor-polynomials,conj:rev-polynomials} would imply the following
conjectures on the exact limits of the normalized minimal stretch factors.

\begin{conj}\label{conj:nonor_limit}
  \begin{displaymath}
    \lim_{\substack{g \to \infty \\ g\mbox{\scriptsize\ even}}}
    (\delta^+(N_{g}))^{g} = (1+\sqrt{2})^2 = (\mbox{silver ratio})^2.
  \end{displaymath}
\end{conj}

\begin{conj}\label{conj:rev-limit}
  \begin{displaymath}
    \lim_{\substack{g \to \infty \\ g\mbox{\scriptsize\ odd}}}
    (\delta^+_{rev}(S_{g}))^{g} = 1+\sqrt{2} = \mbox{silver ratio}.
  \end{displaymath}
\end{conj}

In order to compare \Cref{conj:nonor_limit,conj:rev-limit} to the
orientation-preserving case, we recall that Hironaka asked in \cite[Question
1.12]{Hironaka10} whether
\begin{equation}\label{eq:or-limit}
  \lim_{g\to \infty} (\delta^+(S_g))^g = \left(\frac{1+\sqrt{5}}{2}\right)^2 =
  (\mbox{\emph{golden ratio}})^2.
\end{equation}
Since any pseudo-Anosov map on $N_{g+1}$ can be lifted to a
pseudo-Anosov map on $S_g$ with the same stretch factor, it is natural that the
limit in \Cref{conj:nonor_limit} is larger than the limit in
(\ref{eq:or-limit}). The fact that the limit in \Cref{conj:rev-limit} is
smaller than the limit in (\ref{eq:or-limit}) is consistent with the fact that
nonorientable hyperbolic 3-manifolds can have smaller volume than orientable
ones. For example, the smallest volume non-compact hyperbolic 3-manifold is the
Gieseking manifold, a nonorientable manifold \cite{Adams87}. Since the stretch
factor is related to the volume of the mapping torus \cite{KojimaMcShane18}, on
a fixed surface one can expect to find orientation-reversing pseudo-Anosov maps
with smaller stretch factor than orientation-preserving ones.

\subsection*{Asymptotics along other genus sequences}
\label{sec:asymptotics-other-seqeuences}

We expect the limits in \Cref{conj:nonor_limit,conj:rev-limit} to be different
for other genus sequences. For example, we conjecture the following.

\begin{conj}\label{conj:second-subsequence}
  \begin{displaymath}
    \liminf_{\substack{g \to \infty \\ g\mbox{\scriptsize\ odd}}}
    (\delta^+(N_{g}))^{g} > (1+\sqrt{2})^2 = (\mbox{silver ratio})^2.
  \end{displaymath}
\end{conj}

This conjecture is supported by the following result in the paper
\cite{LiechtiStrennerPenner}. In that paper, we show that if $\delta_{P}(N_g)$
denotes the minimal stretch factor among pseudo-Anosov mapping classes on $N_g$
obtained from Penner's construction, then the sequence $\delta_P(N_g)$ has
exactly two accumulation points as $g \to \infty$. One accumulation
point,~$(1+\sqrt{2})^2$, is the limit for the sequence restricted to even $g$.
The other accumulation point, conjectured to be the largest root of
$x^4-8x^3+13x^2-8x+1$, which is strictly greater than~$(1+\sqrt{2})^2$, is the
limit for the sequence restricted to odd $g$. We expect this dichotomy to be
indicative how the sequence $(\delta^+(N_{g}))^{g}$ behaves for odd and even
genus sequences, respectively, since so far no pseudo-Anosov mapping class of a
nonorientable surface is known to not have a power arising from Penner's
construction (compare with \Cref{question:penners-conj-nonor} below).

\subsection*{Uniformity of minimizing examples}
In the orientation-preserving case, the concrete descriptions of the examples
are all very different. For $g=2$, Zhirov describes the example by the induced
homomorphism~$\pi_1(S) \to \pi_1(S)$. Lanneau and Thiffeault \cite[Appendix
C]{LanneauThiffeault11} describe the same example as a product of the Humphries
generators. For $g=3,4$, Lanneau and Thiffeault use Rauzy--Veech induction, and
for $g=5$, Leininger uses Thurston's construction. While Hironaka gives a
unified construction in \cite{Hironaka10} using fibered face theory, her work
does not give an explicit description of the maps.

In contrast, the descriptions of our examples are explicit and uniform: all of
our examples are constructed as a composition of a Dehn twist and a finite
order mapping class. As we will explain shortly, such constructions cannot
work in the orientable setting.

We remark that it is also possible to construct the examples in
\Cref{theorem:stretch_factors_nonor} and \Cref{theorem:stretch_factors_rev} by
studying fibrations of certain small volume nonorientable hyperbolic
3-manifolds, although we will not discuss this construction in this paper.

\subsection*{Galois conjugates and Penner's construction}

All of our examples have a power that arises from Penner's construction of
pseudo-Anosov mapping classes. In sharp contrast, none of the classical minimal
stretch factor examples have a power that arises from Penner's construction.
This is because these stretch factors have Galois conjugates on the unit
circle. However, Shin and the second author showed in \cite{ShinStrenner15}
that examples with this property do not have a power arising from Penner's
construction.

One may wonder what the reason of this discrepancy is. A heuristic reason for
why Galois conjugates of small stretch factors \emph{should} lie on the unit
circle is that every pseudo-Anosov stretch factor $\lambda$ is a bi-Perron
algebraic unit: a real number larger than~1 whose Galois conjugates 
lie in the annulus $\lambda^{-1} \le |z| \le \lambda$. If $\lambda$
is close to 1, this annulus is a thin neighborhood of the unit circle, so it
seems natural for the Galois conjugates to lie on the unit circle.

However, in \Cref{sec:poly-properties} we will prove the following theorem that
explains why the nonorientable cases are different.

\begin{thm}\label{thm:Galois-nonor}
  If $f$ is a pseudo-Anosov map on a nonorientable surface or an
  orientation-reversing pseudo-Anosov map on an orientable surface, then the
  stretch factor of $f$ does not have Galois conjugates on the unit circle.
\end{thm}

\subsection*{Penner's conjecture on nonorientable surfaces}

Penner asked in \cite{Penner88} whether every pseudo-Anosov map has a power
that arises from his construction.\footnote{The conjecture that this is true is
  known colloquially as Penner's conjecture. However, from the writing in
  \cite[p.~195]{Penner88}, it is unclear whether Penner intended to pose this
  as a question or a conjecture, or even whether he conjectured the opposite.}
This was answered in the negative by Shin an the second author in
\cite{ShinStrenner15} by providing the obstruction mentioned earlier: if the
stretch factor has a Galois conjugate on the unit circle, the pseudo-Anosov map
cannot have a power arising from Penner's construction.

However, \Cref{thm:Galois-nonor} demonstrates that this obstruction is vacuous
for nonorientable surfaces and for orientation-reversing maps. Since there are
no other known obstructions, it is possible that the answer to Penner's
question is in fact ``yes'' in these settings. Some evidence for this is
provided by the fact that all the minimal stretch factor examples we give in
\Cref{theorem:stretch_factors_nonor,theorem:stretch_factors_rev} have a power arising from Penner's
construction. Some evidence against is provided by the failure of the second author
in \cite[Section 7]{StrennerDegrees} to construct certain pseudo-Anosov maps on
nonorientable surfaces using Penner's construction. 

\begin{que}\label{question:penners-conj-nonor}
  Does every pseudo-Anosov map on a nonorientable surface have a power arising
  from Penner's construction?
\end{que}

\begin{que}\label{question:penners-conj-rev}
  Does every orientation-reversing pseudo-Anosov map on an orientable surface
  have a power arising from Penner's construction?
\end{que}

\subsection*{Outline of the paper}
\label{sec:lower-bound-intro}

In
\Cref{sec:construction-nonor,sec:construction-rev-odd},
we construct the examples for
\Cref{theorem:stretch_factors_nonor,theorem:stretch_factors_rev}. This is done by a
generalization of the construction we gave for the Arnoux--Yoccoz pseudo-Anosov
maps in \cite{LiechtiStrennerAY}.

In \Cref{sec:poly-properties}, we give various properties that the
characteristic polynomials of the action on homology have to satisfy for maps
on nonorientable surfaces and orientation-reversing
maps. We also give the proof of \Cref{thm:Galois-nonor} here.

To show that our examples have minimal stretch factor, we follow Lanneau and
Thiffeault's approach for orientable surfaces
\cite{LanneauThiffeault11,LanneauThiffeault11a}: we run a brute-force search
for integral polynomials whose largest root is smaller than our candidate for
the minimal stretch factor and hope that we do not find any. Aside from some
low genus cases, this search is computer-assisted. Our code can be found at
\href{https://github.com/b5strbal/polynomial-filtering}{\texttt{https://github.com/b5strbal/polynomial-filtering}}.

In \Cref{sec:poly-elimination}, we describe this polynomial elimination process
and prove \Cref{theorem:stretch_factors_nonor} without computer assistance in
the case $g=3$. This elimination process ends up being significantly cleaner
for us than it was for Lanneau and Thiffeault. In their case, the restrictions
on the polynomials alone are not sufficient to rule out all polynomials, so
they were left with a few polynomials that needed to be ruled out by studying
the possible singularity structures of the pseudo-Anosov maps and by using
Lefschetz number arguments. For us, no arguments like these are necessary.

\subsection*{Acknowledgements}

We are grateful to Jean-Luc Thiffeault for sharing the code that was used for
the papers \cite{LanneauThiffeault11,LanneauThiffeault11a}. We also thank Dan
Margalit, Mehdi Yazdi and an anonymous referee for helpful comments on an 
earlier version of this paper. The first author was supported by the Swiss National 
Science Foundation~(grant no.~175260)

\section{Construction of pseudo-Anosov maps on nonorientable surfaces}
\label{sec:construction-nonor}

In this and the next section, we use Penner's construction to construct 
pseudo-Anosov mapping classes. We briefly recall Penner's
construction below, stating it in a way that works both for orientable and for 
nonorientable surfaces. For more details, see
\cite[Section 4]{Penner88} or \cite[Section 2]{StrennerDegrees}.

In Penner's construction, we have a collection of two-sided simple closed
curves $C = \{c_1,\ldots,c_n\}$ that fill the surface (the complement of the
curves is a union of disks and once-punctured disks), that pairwise intersect minimally,
and that are \emph{marked inconsistently}. This means that there is a small regular 
neighborhood $N(c_i)$ for each curve $c_i$ and an orientation of each annulus $N(c_i)$ 
such that the orientation of $N(c_i)$ and~$N(c_j)$ are different at each
intersection whenever $i\ne j$. Penner showed that any product of the Dehn
twists $T_{c_i}$ is pseudo-Anosov assuming that
\begin{itemize}
\item each twist $T_{c_i}$ is right-handed according to the orientation of
  $N(c_i)$,
\item each twist $T_{c_i}$ is used in the product only with positive powers,
\item each twist $T_{c_i}$ is used in the product at least once.
\end{itemize}

Note that if the surface is oriented, then the above conditions in Penner's 
construction say that the collection of curves is a union of two multicurves~$\Gamma_1$ 
and~$\Gamma_2$, and the Dehn twists along the curves in~$\Gamma_1$ are 
all right-handed, whereas the Dehn twists along the curves in~$\Gamma_2$ are 
all left-handed with respect to the orientation of the surface.  

We will present the construction of our examples as follows. First we define the rotationally
symmetric graphs that will be the intersection graphs of the collections of
curves. Then we describe the rotationally symmetric surfaces and curves on
these surfaces whose intersection matrices realize the given graphs. Finally,
we define our mapping classes as a composition of a Dehn twist and a rotation.

\subsection{The graphs}
Let $k$ and $n$ be integers of different parity such that $n\ge 3$ and
$1 \le k \le n-1$. Let~$G_{n,k}$ be the graph whose vertices are the vertices
of a regular $n$-gon and every vertex $v$ is connected to the $k$ vertices that
are the farthest away from $v$ in the cyclic order of the vertices.

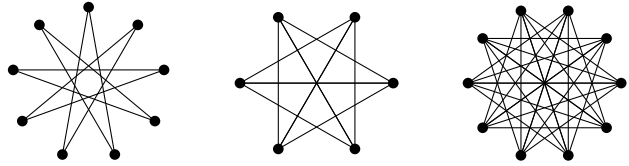
\begin{figure}[htb]
  \centering
  \begin{tikzpicture}
    \begin{scope}[xshift=-6cm]
    \node[draw=none,minimum size=2cm,regular polygon,regular polygon sides=9] (a) {};
    \foreach \x in {1,2,...,9}{
      \fill (a.corner \x) circle[radius=2pt];
      \foreach \y in {5}{
        \pgfmathtruncatemacro{\j}{mod(\x+\y-1,9)+1}
        \draw (a.corner \x) -- (a.corner \j);
      }
    }
    \end{scope}

    \begin{scope}[xshift=-3cm]
    \node[draw=none,minimum size=2cm,regular polygon,regular polygon sides=6] (a) {};
    \foreach \x in {1,2,...,6}{
      \fill (a.corner \x) circle[radius=2pt];
      \foreach \y in {2,3}{
        \pgfmathtruncatemacro{\j}{mod(\x+\y-1,6)+1}
        \draw (a.corner \x) -- (a.corner \j);
      }
    }
    \end{scope}

    \node[draw=none,minimum size=2cm,regular polygon,regular polygon sides=10] (a) {};
    \foreach \x in {1,2,...,10}{
      \fill (a.corner \x) circle[radius=2pt];
      \foreach \y in {3,4,5}{
        \pgfmathtruncatemacro{\j}{mod(\x+\y-1,10)+1}
        \draw (a.corner \x) -- (a.corner \j);
      }
    }

  \end{tikzpicture}
  \caption{The graphs $G_{9,2}$, $G_{6,3}$ and $G_{10,5}$.}
  \label{fig:surface-and-graphs}
\end{figure}

\subsection{The surfaces}
\label{sec:surfaces}

For each $G_{n,k}$, we will construct a nonorientable surface $\Sigma_{n,k}$ that contains 
a collection of curves with intersection graph $G_{n,k}$. To construct $\Sigma_{n,k}$, start with 
a disk with one crosscap. By this, we mean that we cut a smaller disk out of the disk and identify 
the antipodal points of the boundary of the small disk. We indicate this identification with a cross 
inside the small disk, see \Cref{fig:surface}. The resulting surface is homeomorphic to the 
M\"obius strip.

Next, we consider $2n$ disjoint intervals on the boundary of the disk and label the intervals with 
integers from 1 to $n$ so that each label is used exactly twice. In the cyclic order, the labels are 
$1, s, 2, s+1, \ldots, n, s+n$, where $s =\frac{n+k+3}2$ and all labels are understood modulo $n$.

For each label, the corresponding two intervals are connected by a twisted strip, as on \Cref{fig:surface}.

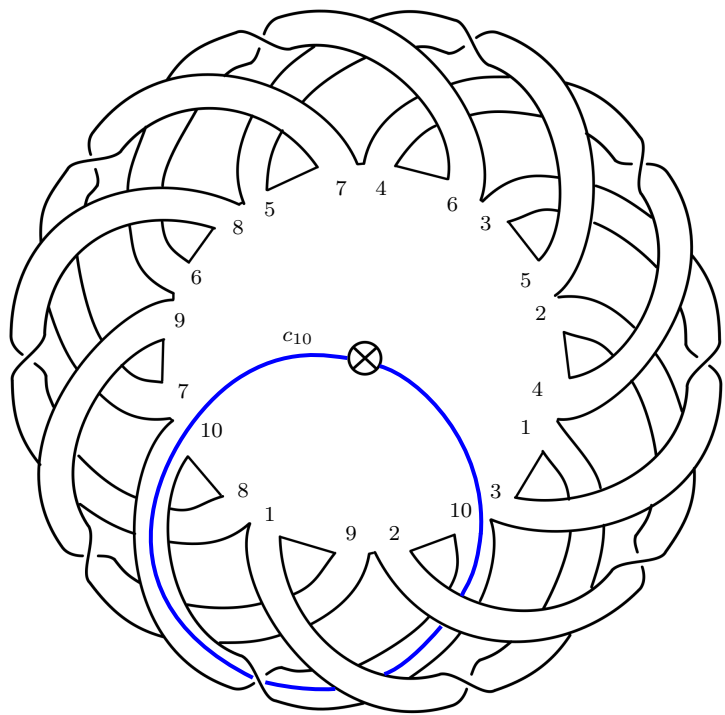
\begin{figure}[htb]
\labellist
\scriptsize\hair 2pt
 \pinlabel {$1$} at 198 163
 \pinlabel {$9$} at 247 152
 \pinlabel {$2$} at 273 152
 \pinlabel {$10$} at 313 166
 \pinlabel {$3$} at 334 177
 \pinlabel {$1$} at 352 216
 \pinlabel {$4$} at 359 239
 \pinlabel {$2$} at 361 285
 \pinlabel {$5$} at 352 305
 \pinlabel {$3$} at 328 340
 \pinlabel {$6$} at 308 351
 \pinlabel {$4$} at 265 361
 \pinlabel {$7$} at 241 361
 \pinlabel {$5$} at 198 348
 \pinlabel {$8$} at 179 337
 \pinlabel {$6$} at 154 307
 \pinlabel {$9$} at 144 281
 \pinlabel {$7$} at 146 238
 \pinlabel {$10$} at 163 214
 \pinlabel {$8$} at 182 178
 \pinlabel {$c_{10}$} at 215 270
\endlabellist
\centering
\includegraphics[width=0.80\textwidth]{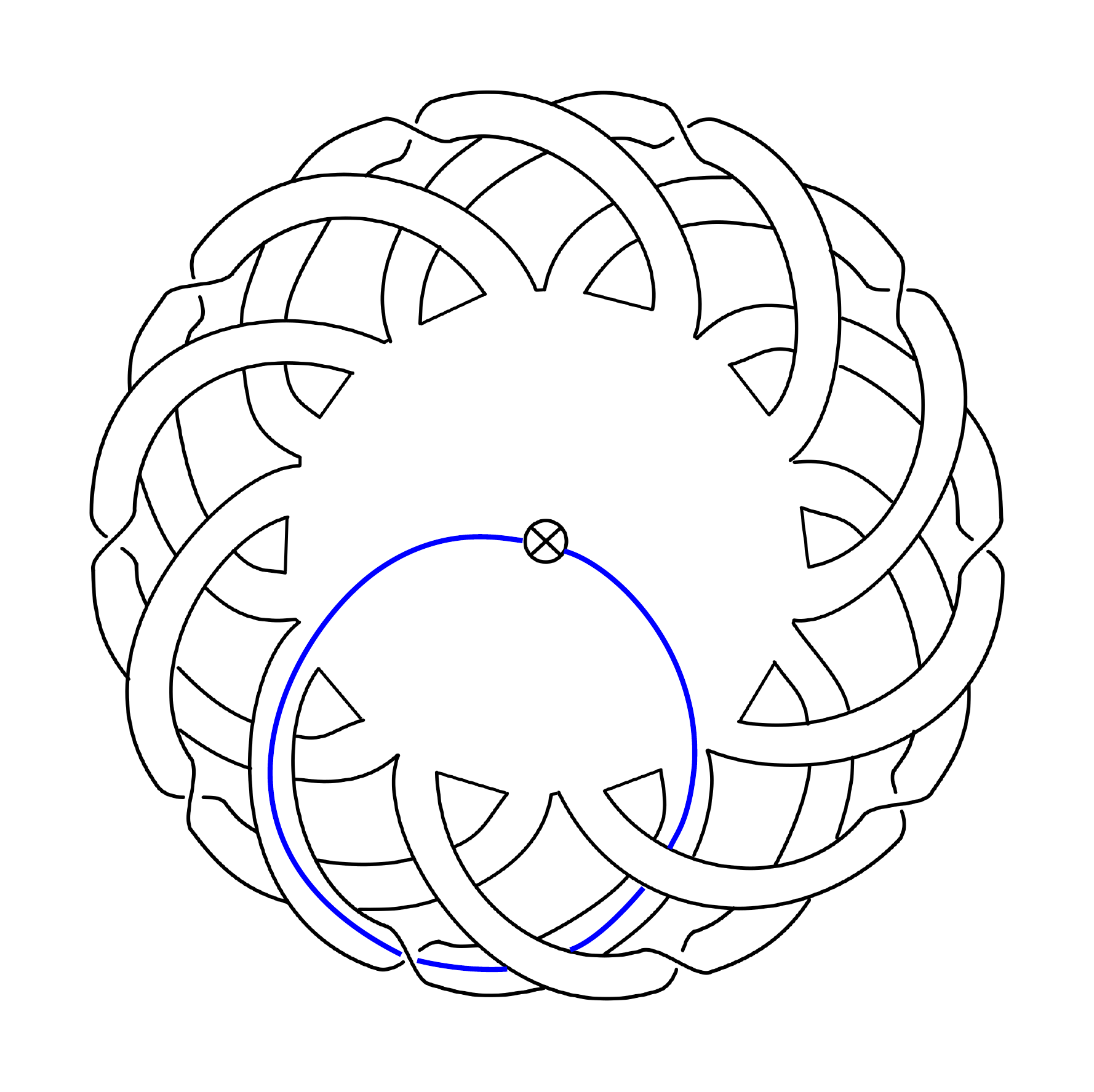}
  \caption{The surface $\Sigma_{10,5}$ and the curve $c_{10}$.}
\label{fig:surface}
\end{figure}

\begin{lemma}\label{lemma:euler-char}
  The Euler characteristic of $\Sigma_{n,k}$ is $-n$.
\end{lemma}
\begin{proof}
  The disk with a crosscap has zero Euler characteristic (it is homeomorphic to a M\"obius strip), and each attached twisted strip has contribution $-1$.
\end{proof}

\begin{lemma}\label{lemma:bdy-components}
  The number of boundary components of $\Sigma_{n,k}$ is $\gcd(n,k)$.
\end{lemma}
\begin{proof}
  We will show that the number of boundary components of $\Sigma_{n,k}$ is the
  same as the number of orbits of the dynamical system $x \mapsto x+n-k$ in the
  group $\Z/2n\Z$. The number of such orbits is $\gcd(n-k, 2n) = \gcd(k,n)$,
  since $n-k$ is odd.

  To prove our claim, we identify $\Z/2n\Z$ with the $2n$ intervals in the
  cyclic order. We claim that the right endpoint of the interval at position
  $i$ lies on the same boundary component as the right endpoint of the interval
  at position $i+n-k$. One can see this by induction. In the case $k = n-1$,
  the cyclic order of labels is $1, 1, \ldots, n, n$, so the twisted strips
  identify the right endpoint of every interval with the right endpoint of the
  next interval. When $k=n-3$, the cyclic order is
  $1, n, 2, 1, \ldots, n, n-1$, in which case every third right endpoint is on
  the same boundary component, and so on.
\end{proof}

\begin{proposition}\label{prop:nonor-surface-type}
  The surface $\Sigma_{n,k}$ is homeomorphic to the nonorientable surface of genus $n-\gcd(k,n)+2$ with $\gcd(k,n)$ boundary components.
\end{proposition}
\begin{proof}
  The Euler characteristic of the nonorientable surface of genus $g$ with $b$
  boundary components is $2-g-b$. By
  \Cref{lemma:euler-char,lemma:bdy-components}, we obtain the equation
  $2-g-\gcd(k,n) = -n$. Rearranging, we obtain $g = n-\gcd(k,n)+2$.
\end{proof}

\subsection{The curves}
\label{sec:curves}

We construct a two-sided curve $c_i$ for each label $i = 1, \ldots, n$ as
follows. Each curve consists of two parts. One part of each curve is the core
of the strip corresponding to the label. The other part is an arc inside the
disk that passes through the crosscap and connects the corresponding two
intervals. The curve $c_{10}$ is shown on \Cref{fig:surface}.

Note that every pair of curves intersects either once or not at all. The curves
$c_i$ and $c_j$ are disjoint if and only if the two $i$ labels and the two $j$
labels \emph{link} in the cyclic order. In other words, if the two $i$ labels
separate the two $j$ labels.

\begin{lemma}
  The intersection graph of the curves $c_i$ on $\Sigma_{n,k}$ is $G_{n,k}$.
\end{lemma}
\begin{proof}
  We proof the lemma by induction. If $k=n-1$, then $s=1$, so the cyclic order is $1, 1, 2, 2, \ldots, n, n$. Since the no two labels link, all pairs of curves intersect and the intersection graph is the complete graph $G_{n,n-1}$.

  Now suppose $k$ is decreased by 2. Then $s$ is decreased by 1, and we obtain the cyclic order $1, n, 2, 1, 3, 2, \ldots, n, n-1$. As a consequence, 1 becomes linked with 2 and $n$. Hence the intersection graph is indeed $G_{n,k}$.

  It is easy to see that every time $k$ is decreased by two each label is linked with two more labels, hence the intersection graph is always $G_{n,k}$.
\end{proof}

\begin{lemma}
  The curves $c_i$ can be marked so that all intersections are inconsistent.
\end{lemma}
\begin{proof}
  Choose markings for the $c_i$ which are invariant under the rotational
  symmetry, see \Cref{fig:curves}. The marking of the curves is indicated by
  the coloring as follows. Consider the orientable surface obtained by removing
  the crosscap and cutting the strips attached to the disk in the middle.
  Choose an orientation of this surface. Then color the arcs composing the
  curves using red and blue depending on whether the orientation of the
  neighborhood of the curve matches the orientation of the surface or not. Note
  that the color of a curve changes when it goes through the crosscap or the
  middle of a strip.

  Since blue and red meets at every intersection, the marking is inconsistent.
\end{proof}

\begin{figure}[ht]
  \centering
  \includegraphics[width=0.72\textwidth]{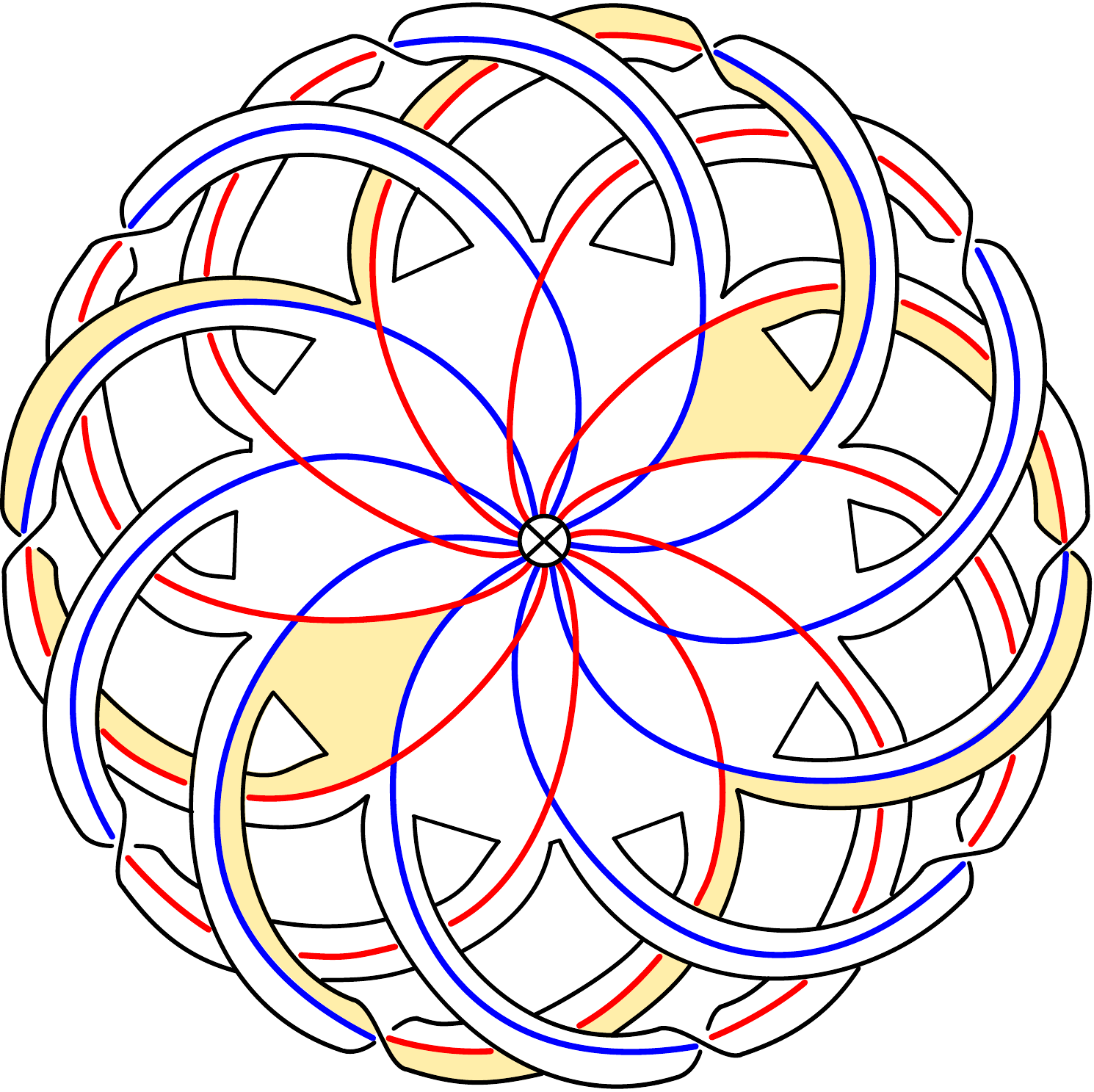}
  \caption{A collection of filling inconsistently marked curves.}
  \label{fig:curves}
\end{figure}

\subsection{The mapping classes}
\label{sec:mapping-classes}

Denote by $r$ the rotation of $\Sigma_{n,k}$ by one click in the clockwise direction. Define the mapping class
\begin{displaymath}
  f_{n,k} = r\circ T_{c_1}
\end{displaymath}
where $T_{c_1}$ is a Dehn twist about the curve $c_1$. (There are two possible
directions for the Dehn twist, but either choice works for our purposes.) Note that
\begin{displaymath}
  f_{n,k}^n = T_{c_n} \circ \cdots \circ T_{c_1},
\end{displaymath}
so $f_{n,k}^n$ arises from Penner's construction. In particular, $f_{n,k}^n$ is pseudo-Anosov and so is~$f_{n,k}$.

We remark that for~$k=n-1$, the mapping class~$f_{n,k}$ coincides with the
nonorientable Arnoux-Yoccoz mapping class~$h_{n-1}$, described as a product of
a Dehn twist and a finite order mapping class by the authors
in~\cite{LiechtiStrennerAY}.

\begin{proposition}\label{prop:nonor-defining-poly}
  The stretch factor of $f_{n,k}$ is the largest root of $x^n - x^{n-r} - x^{n-r-1} -\cdots - x^{r+1} - x^{r} - 1$, where $r = \frac{n-k+1}2$.
\end{proposition}
\begin{proof}
  To compute the stretch factor, we use Penner's approach in the section titled
  ``An upper bound by example'' in \cite{Penner91}. Penner constructed an
  invariant bigon track by smoothing out the intersections of the curves $c_i$.
  Each $c_i$ defines a characteristic measure $\mu_i$ on this bigon track,
  defined by assigning 1 to the branches traversed by $c_i$ and zero to the
  rest. The cone generated by the $\mu_i$ is invariant under both $T_{c_1}$ and
  $r$, hence it contains the unstable foliation, and the stretch factor is
  given by the largest eigenvalue of the action of $r\circ T_{c_1}$ on this
  cone. The rotation $r$ acts by a permutation matrix and the matrix
  corresponding to $T_{c_1}$ is the sum of the identity matrix and matrix
  obtained by the intersection matrix $i(C,C)$ by zeroing out all rows except
  the first row. The product of these two matrices takes the following form:
\begin{displaymath}
  \begin{pmatrix}
    0 & 1 & 0 & 0 & 0 & 0 & 0 & 0 & 0 & 0 \\
    0 & 0 & 1 & 0 & 0 & 0 & 0 & 0 & 0 & 0 \\
    0 & 0 & 0 & 1 & 0 & 0 & 0 & 0 & 0 & 0 \\
    0 & 0 & 0 & 0 & 1 & 0 & 0 & 0 & 0 & 0 \\
    0 & 0 & 0 & 0 & 0 & 1 & 0 & 0 & 0 & 0 \\
    0 & 0 & 0 & 0 & 0 & 0 & 1 & 0 & 0 & 0 \\
    0 & 0 & 0 & 0 & 0 & 0 & 0 & 1 & 0 & 0 \\
    0 & 0 & 0 & 0 & 0 & 0 & 0 & 0 & 1 & 0 \\
    0 & 0 & 0 & 0 & 0 & 0 & 0 & 0 & 0 & 1 \\
    1 & 0 & 0 & 1 & 1 & 1 & 1 & 1 & 0 & 0 \\
  \end{pmatrix}
\end{displaymath}
This particular matrix belongs to $f_{10,5}$.

This matrix is the companion
matrix of the polynomial in the statement of the proposition. Hence the
characteristic polynomial of this matrix is indeed that polynomial.
\end{proof}

Our next goal is to determine the singularity structure of the mapping classes
$f_{n,k}$. For this, first we need a lemma.

Consider the complementary regions of the curves $\{c_1,\ldots,c_n\}$. There
are two types of regions depending on whether a region contains a boundary
component of $\Sigma_{n,k}$ (type~1) or not (type~2). A region of type~1 is an
annulus that is bounded by a boundary component $\beta$ of $\Sigma_{n,k}$ on one
side and by a polygonal path consisting of arcs of the curves $c_i$ on the
other side. The shaded region on \Cref{fig:curves} illustrates a region of type~1.

\begin{lemma}\label{lem:polygon-length}
  The length of these polygonal paths is $\frac{4n}{\gcd(k,n)}$.
\end{lemma}
\begin{proof}
  This follows from the observation that every point in the orbit in $\Z/2n\Z$
  corresponding to the boundary component $\beta$ (see the proof of
  \Cref{lemma:bdy-components}) has two associated arcs. Since the number of
  orbits is~$\gcd(k,n)$, the length of each orbit is~$\frac{2n}{\gcd(k,n)}$,
  and hence the length of the polygonal path is twice this quantity.
\end{proof}

\begin{proposition}\label{prop:singularity-type-nonor}
  The pseudo-Anosov mapping class $f_{n,k}$ has $gcd(k,n)$ singularities, one for each boundary component. The number of prongs of each singularity is $\frac{2n}{\gcd(k,n)}$.
\end{proposition}
\begin{proof}
  Each complementary region of the curves $\{c_1,\ldots,c_n\}$ contains either
  one singularity or none. The number of prongs of a singularity equals the
  number of cusps of the bigon track obtained by the smoothing process that are
  contained in the same region as the singularity. If the number of cusps is 2,
  then the region does not contain a singularity. If the number of cusps is
  $k>2$, then it contains a $k$-pronged singularity.

  Regions of type 2 are rectangles (bounded by four subarcs of the curves
  $c_i$), and hence contain two cusps. So they do not correspond to
  singularities.

  The lengths of the polygonal paths bounding regions of type 1 are
  $\frac{4n}{\gcd(k,n)}$ by \Cref{lem:polygon-length}, so the number of cusps
  in these regions is $\frac{2n}{\gcd(k,n)}$. Therefore the singularities have
  that many prongs. By \Cref{lemma:bdy-components}, the number of such regions
  is $\gcd(k,n)$, so that is also the number of the singularities.
\end{proof}

As a corollary of
\Cref{prop:nonor-defining-poly,prop:singularity-type-nonor,prop:nonor-surface-type},
we have the following.

\begin{cor}\label{cor:construction-nonor}
  There exist pseudo-Anosov mapping classes with an orientable invariant
  foliation on the surfaces $N_g$ with the data below. All of these examples
  belong to the family $f_{n,k}$ for the $n$ and $k$ shown in the table.
  \begin{center}
    \begin{tabular}[ht]{l|l|l|l|l|l}
      $g$ & $n$ & $k$ & $\lambda(f_{n,k})$ & minimal polynomial & singularity type\\
      \hline
      4* & 3 & 2 & 1.83929 & $x^3 - x^2 - x - 1$ & (6)\\
      5* & 6 & 3 & 1.51288 & $x^4 - x^3 - x^2 + x - 1$ & (4,4,4)\\
      6* & 5 & 2 & 1.42911 & $x^5 - x^3 - x^2 - 1$ & (10)\\
      7* & 10 & 5 & 1.42198 & $x^6 - x^5 - x^3 + x - 1$ & (4,4,4,4,4) \\
      8* & 7 & 2 & 1.28845 & $x^7 - x^4 - x^3 - 1$ & (14)\\
      9 & 8 & 3 & 1.35680 & $x^8 - x^5 - x^4 - x^3 - 1$ & (16) \\
      10* & 9 & 2 & 1.21728 & $x^9 - x^5 - x^4 - 1$ & (18)\\
      11 & 12 & 3 & 1.22262 & $\frac{x^{12} - x^{7} - x^{6} - x^{5} - 1}{x^{2}
                              + x + 1}$ & (8,8,8)\\
      12* & 11 & 2 & 1.17429 & $x^{11} - x^6 - x^5 - 1$  & (22)\\
      13 & 22 & 11 & 1.27635 & $x^{12} - x^{11} - x^6 + x - 1$ & (4$^{11}$) \\
      14* & 13 & 2 & 1.14551 & $x^{13} - x^7 - x^6 - 1$ & (26)\\
      15 & 14 & 3 & 1.18750 & $x^{14} - x^{8} - x^{7} - x^{6} - 1$ & (28) \\
      16* & 15 & 2 & 1.12488 & $x^{17} - x^9 - x^8 - 1$ & (30)\\
      17 & 18 & 3 & 1.14259 & $\frac{x^{18} - x^{10} - x^{9} - x^{8} - 1}{x^{2}
                              + x + 1}$ & (12,12,12) \\
      18* & 17 & 2 & 1.10938 & $x^{19} - x^{10} - x^9 - 1$ & (34)\\
      19 & 18 & 5 & 1.20514 & $x^{18} - x^{11} - x^{10} - x^{9} - x^{8} - x^{7} - 1$
         & (36) \\
      20* & 19 & 2 & 1.09730 & $x^{23} - x^{12} - x^{11} - 1$ & (38)\\
    \end{tabular}
  \end{center}
  (4$^{11}$ means that there are 11 singularities with 4 prongs.)
\end{cor}

In each genus, the family $f_{n,k}$ contains several examples. In the table
above, we have listed only the example with the smallest stretch factor. In the
starred cases, we will be able to certify that the given stretch factors are not
only minimal in the family $f_{n,k}$ but among all pseudo-Anosov maps with
an orientable invariant foliation.

\section{Orientation-reversing pseudo-Anosov mapping classes on odd genus surfaces}
\label{sec:construction-rev-odd}

In this section, we construct an orientation-reversing pseudo-Anosov mapping
class with small stretch factor on every odd genus orientable surface. The construction is
analogous to the construction in the previous section, but simpler. As in the
previous section, we separate the construction of the surfaces, the curves and
finally the mapping classes.

\subsection{The surfaces}

For every $k\ge 2$, consider the surface $\Sigma_k$ obtained by chaining together $2k$ annuli in a cycle as on
\Cref{fig:even-cycle}.

\begin{figure}[ht]
  \centering
  \includegraphics[width=0.60\textwidth]{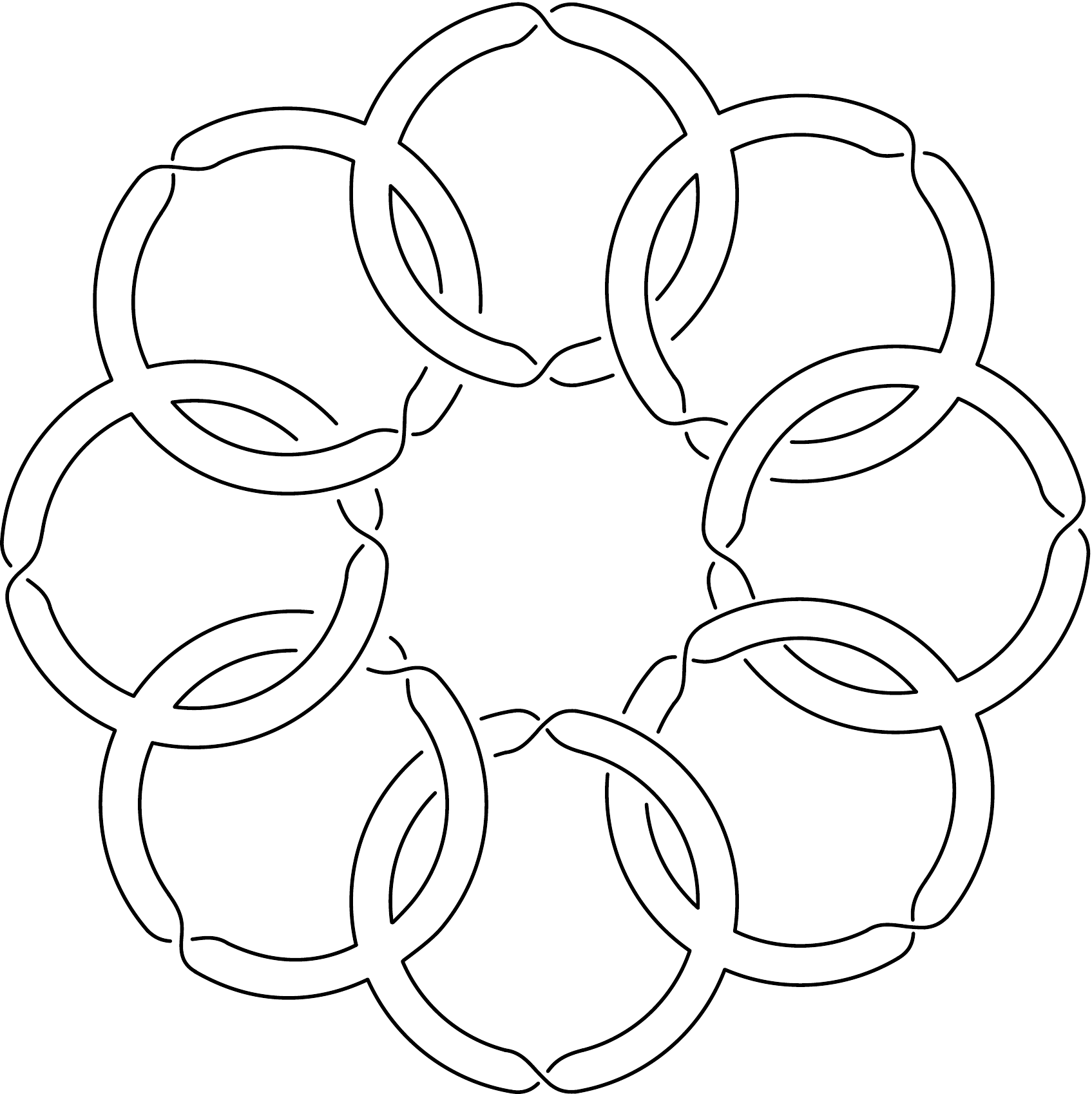}
  \caption{The surface~$\Sigma_k$.}
  \label{fig:even-cycle}
\end{figure}

\begin{proposition}\label{prop:reversing-bdy}
  The number of boundary components of $\Sigma_k$ is 4 if $k$ is even and 2 if~
  $k$ is odd.
\end{proposition}
\begin{proof}
  The boundary of $\Sigma_k$ is composed of $8k$ arcs, 4 arcs for each annulus.
  Our goal is to determine which of them belong to the same boundary component.

  Denote by $r$ the rotation of $\Sigma_k$ by one click. By tracing the
  boundary, one can see that every boundary point $x$ lies on the same boundary
  component as $r^4(x)$. Moreover, the path between $x$ and $r^4(x)$ traverses
  each of the 4 types of arcs exactly once. Therefore it suffices to pick any
  boundary point $x$ and determine into how many equivalence classes the set
  $\{x,r(x),\ldots, r^{2k-1}(x)\}$ falls apart. The number of such equivalence
  classes is 4 if $k$ is even and 2 if $k$ is odd.
\end{proof}

\begin{proposition}\label{prop:reversing-surface-type}
  The surface $\Sigma_k$ is homeomorphic to $S_{k-1,4}$ if $k$ is even and $S_{k,2}$
  if $k$ is odd.
\end{proposition}
\begin{proof}
  We have $\chi(\Sigma_k) = 2k$. From the equation $\chi = 2-2g-b$, where $g$ is
  the genus and $b$ is the number of boundary components, it follows that $g =
  k+1-\frac{b}2$. The statement now follows from \Cref{prop:reversing-bdy}.
\end{proof}

As a consequence, the construction only produces odd genus examples.

\subsection{The curves}

From now on, suppose that $k$ is even. Consider the set
$C = \{c_1,\ldots,c_{2k}\}$ of core curves of the $2k$ annuli. Our numbering
will differ from the standard cyclic numbering; we will explain this shortly.
As in \Cref{sec:curves}, any rotationally symmetric marking of the curves is an
inconsistent marking.

The intersection graph of $C$ is a cycle of length $2k$. We draw this
cycle as on \Cref{fig:anternative-cycle}: the vertices are the vertices of a
regular polygon and every vertex is connected to the two vertices that are the
second furthest in the cyclic order. We number the curves according to the
cyclic orientation induced by this picture.
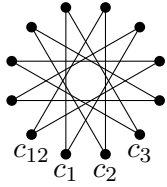
\begin{figure}[ht]
  \centering
  \begin{tikzpicture}
    \node[draw=none,minimum size=2cm,regular polygon,regular polygon sides=12] (a) {};
    \foreach \x in {1,2,...,12}{
      \fill (a.corner \x) circle[radius=2pt];
      \foreach \y in {5}{
        \pgfmathtruncatemacro{\j}{mod(\x+\y-1,12)+1}
        \draw (a.corner \x) -- (a.corner \j);
      }
    }
    \node[below] at (a.corner 7) {$c_1$};
    \node[below] at (a.corner 8) {$c_{2}$};
    \node[below] at (a.corner 9) {$c_3$};
    \node[below] at (a.corner 6) {$c_{12}$};
  \end{tikzpicture}
  \caption{Our unusual way of numbering the curves. For example, the curve $c_1$
    intersects $c_{k}$ and $c_{k+2}$, not $c_2$ and $c_{2k}$.}
  \label{fig:anternative-cycle}
\end{figure}

\subsection{The mapping classes}

Denote by $r$ the rotation of $\Sigma_k$ (see \Cref{fig:even-cycle}) by one
click in the clockwise direction. Since $c_i$ and $r(c_i)$ intersect for all
$i$, the rotation $r$ induces a rotation of the cycle on
\Cref{fig:anternative-cycle} by $k-1$ clicks. So $r^{k-1}$ rotates the cycle by
$(k-1)^2 = k^2-2k+1$ clicks, which is congruent to 1 modulo $2k$ if $k$ is
even. Therefore $r^{k-1}$ induces rotating the cycle on
\Cref{fig:anternative-cycle} by one click (in the clockwise direction, assuming
that we have chosen the numbering of the curves accordingly). In particular, we
have $r^{k-1}(c_{i+1}) = c_i$.

We are now ready to define the mapping class:
\begin{displaymath}
  \psi_{k} = r^{k-1}\circ T_{c_1}.
\end{displaymath}
Note that
\begin{displaymath}
  \psi_k^{2k} = T_{c_{2k}} \circ \cdots \circ T_{c_1},
\end{displaymath}
so $\psi_{k}^{2k}$ arises from Penner's construction. 
In particular, $\psi_{k}^{2k}$ is pseudo-Anosov and so is~$\psi_k$.
Note that while~$\psi_{k}^{2k}$ is orientation-preserving,~$\psi_k$ is orientation-reversing. 
This follows from the observation that~$T_{c_1}$ is orientation-preserving 
and~$r$ is orientation-reversing. 

\begin{proposition}\label{prop:reversing-defining-poly}
  The stretch factor of $\psi_k$ is the largest root of $x^{2k} - x^{k+1} -x^{k-1} - 1$.
\end{proposition}
\begin{proof}
  The proof is similar to the proof of \Cref{prop:nonor-defining-poly}. We have
  \begin{displaymath}
    i(C,C) = \begin{pmatrix}
      0 & 0 & 0 & 1 & 0 & 1 & 0 & 0 \\
      0 & 0 & 0 & 0 & 1 & 0 & 1 & 0 \\
      0 & 0 & 0 & 0 & 0 & 1 & 0 & 1 \\
      1 & 0 & 0 & 0 & 0 & 0 & 1 & 0 \\
      0 & 1 & 0 & 0 & 0 & 0 & 0 & 1 \\
      1 & 0 & 1 & 0 & 0 & 0 & 0 & 0 \\
      0 & 1 & 0 & 1 & 0 & 0 & 0 & 0 \\
      0 & 0 & 1 & 0 & 1 & 0 & 0 & 0
    \end{pmatrix}
    \quad \mbox{and} \quad
    M = \begin{pmatrix}
      0 & 1 & 0 & 0 & 0 & 0 & 0 & 0 \\
      0 & 0 & 1 & 0 & 0 & 0 & 0 & 0 \\
      0 & 0 & 0 & 1 & 0 & 0 & 0 & 0 \\
      0 & 0 & 0 & 0 & 1 & 0 & 0 & 0 \\
      0 & 0 & 0 & 0 & 0 & 1 & 0 & 0 \\
      0 & 0 & 0 & 0 & 0 & 0 & 1 & 0 \\
      0 & 0 & 0 & 0 & 0 & 0 & 0 & 1 \\
      1 & 0 & 0 & 1 & 0 & 1 & 0 & 0
    \end{pmatrix}
  \end{displaymath}
  where $M$ is the matrix of the action of $\psi_k$ on the cone of measures (the
  product of a permutation matrix and the sum of the identity matrix and the
  first row of $i(C,C)$). The matrices above illustrate the case $k=4$. The
  matrix $M$ is the companion matrix of the polynomial in the proposition.
\end{proof}

\begin{proposition}\label{prop:singularity-type-rev}
  The pseudo-Anosov mapping class $\psi_k$ has four $k$-pronged singularities.
\end{proposition}
\begin{proof}
  By \Cref{prop:reversing-bdy} and its proof, each of the four boundary
  components of~$\Sigma_k$ consists of~$2k$ arcs if~$k$ is even. There is a
  prong for every second corner of the boundary path, therefore there are $k$
  prongs for each singularity.
\end{proof}

\begin{cor}\label{cor:construction-reversing}
  There exist orientation-reversing pseudo-Anosov mapping classes with
  orientable invariant foliations on the surfaces $S_g$ with the data below.
  All of these examples belong to the family $\psi_k$ for the $k$ shown in the
  table.
  \begin{center}
    \begin{tabular}[ht]{l|l|l|l|l|l}
      $g$ & $k$ & $\lambda(\psi_k)$ & largest root of & singularity type\\
      \hline
      1 & 2 & 1.61803 & $x^2 - x - 1 = \frac{x^4-x^3-x-1}{x^2+1}$ & no singularities\\
      3 & 4 & 1.25207 & $x^8 - x^5 - x^3 - 1$ & (4,4,4,4) \\
      5 & 6 & 1.15973 & $x^{12} - x^7 - x^5 - 1$ & (6,6,6,6) \\
      7 & 8 & 1.11707 & $x^{16} - x^9 - x^7 - 1$ & (8,8,8,8) \\
      9 & 10 & 1.09244 & $x^{20}- x^{11} - x^9 - 1$ & (10,10,10,10)\\
      11 & 12 & 1.07638 & $x^{24} - x^{13} - x^{11} - 1$ & (12,12,12,12) \\
    \end{tabular}
  \end{center}
\end{cor}
\begin{proof}
  The statement follows from
  \Cref{prop:reversing-defining-poly,prop:singularity-type-rev,prop:reversing-surface-type}.
  The reason we have no singularities in the genus 1 case is that by
  \Cref{prop:singularity-type-rev} the ``singularities'' have two
  prongs, so they are not actually singularities.
\end{proof}

We remark that in the genus 1, there is an alternative, simpler construction that yields the
same stretch factor. Consider the matrix
\begin{displaymath}
  M =
  \begin{pmatrix}
    0 & 1 \\
    1 & 1 \\
  \end{pmatrix}
\end{displaymath} with determinant $-1$. The corresponding linear map $\R^2\to
\R^2$ maps $\Z^2$ to $\Z^2$, hence it descends to an Anosov diffeomorphism $f$ of
the torus $\R^2/\Z^2$. Its stretch factor is the largest root of the $x^2-x-1$,
the characteristic polynomial of $M$.

\section{Restrictions on polynomials}
\label{sec:poly-properties}

Pseudo-Anosov stretch factors are roots of integral polynomials. The properties
of these integral polynomials are similar, but slightly different depending on
whether a pseudo-Anosov mapping class is an orientation-preserving or
orientation-reversing mapping class on an orientable surface or a mapping class
on a nonorientable surface. In this section, we discuss these properties for
nonorientable surfaces and orientation-reversing mapping classes.

A polynomial $p(x)$ of degree $n$ is called \emph{reciprocal} if $p(x)
= \pm x^n p(x^{-1})$, in other words, when its coefficients are the same
in reverse order up to sign. Analogously, we define~$p(x)$ to be
\emph{skew-reciprocal} if $p(x) = \pm x^n p(-x^{-1})$.

\begin{proposition}\label{prop:nonor-poly-properties}
  Let $\psi: N_g\to N_g$ be a pseudo-Anosov map with a transversely orientable invariant
  foliation on the closed non-orientable surface $N_g$ of genus $g$. Then its
  stretch factor $\lambda$ is a root of a (not necessarily
  irreducible) polynomial $p(x) \in \Z[x]$ with the following properties:
  \begin{enumerate}
  \item\label{item:degree-nonor} $\deg(p) = g-1$.
  \item\label{item:monic} $p(x)$ is monic and its constant coefficient is $\pm 1$.
  \item\label{item:other-roots-nonor} The absolute values of the roots of
    $p(x)$ other than $\lambda$ lie in the open interval
    $(\lambda^{-1}, \lambda)$. In particular, $p(x)$ is not reciprocal or
    skew-reciprocal.
  \item\label{item:recip-mod-two} $p(x)$ is reciprocal mod 2.
  \end{enumerate}
\end{proposition}
\begin{proof}
  Note that exactly one of the stable and unstable foliations is transversely
  orientable (otherwise the surface itself would be orientable). We will assume
  that it is the stable foliation, otherwise we replace $\psi$ by its inverse.

  Consider the action $\psi^*:H^1(N_g;\R)\to H^1(N_g;\R)$ defined by pullback
  on cohomology with real coefficients. Since the stable foliation is transversely 
  orientable, it is represented by a closed real 1-form, that is, an element of
  $H^1(N_g;\R)$. The stable foliation $\calF^s$ is the one whose leaves are
  contracting and hence the surface is expanding in the transverse direction.
  Therefore the measure of a transverse arc in the pullback $\psi^*(\calF^s)$ is
  $\lambda$ times its measure in $\calF^s$. Hence $\calF^s$ is an eigenvector of
  the map $\psi^*$ with eigenvalue~$\lambda$ or~$-\lambda$.

  Let $p(x)$ be the characteristic polynomial of $\psi^*$. Note that
  $\dim(H^1(N_g,\R)) = g-1$, hence $\deg(p) = g-1$. This proves
  (\ref{item:degree-nonor}).

  The polynomial $p(x)$ has integral coefficients, since $\psi^*$ restricts to
  an action $H^1(N_g;\Z) \to H^1(N_g;\Z)$. This restriction is invertible,
  since the action of $\Mod(N_g)$ on $H^1(N_g;\Z)$ is a group representation,
  so the determinant of $\psi^*$ is $\pm 1$. Therefore the constant coefficient
  of $p(x)$ is $\pm 1$. Also, as a characteristic polynomial, $p(x)$ is monic.
  This proves (\ref{item:monic}).

  It is a standard fact from the theory of orientation-preserving pseudo-Anosov 
  mapping classes on orientable surfaces that stretch factors are strictly maximal among 
  their Galois conjugates. Moreover, in case the invariant foliations are orientable, the 
  spectral radius of the action induced on the first homology equals the stretch factor, and 
  every other eigenvalue is strictly smaller, see, for example,~\cite[Theorem 5.3
  (1)]{McMullen03a}. Applying this result to the orientation-preserving lift 
  $\tilde\psi: S_{g-1} \to S_{g-1}$ of $\psi$ to the orientable double cover 
  $S_{g-1}$ of $N_g$, we obtain that any root $\lambda'$ of
  $p(x)$ other than $\pm \lambda$ satisfies~$|\lambda'| < |\lambda|$. Applying
  the same theorem for $\tilde{\psi}^{-1}$, we conclude that any root
  $\lambda'$ of~$p(x)$ other than~$\pm \lambda^{-1}$ satisfies
  $|\lambda'^{-1}| < |\lambda|$. Therefore absolute values of the roots of~$p(x)$
  other than~$\pm \lambda$ and possibly~$\pm \lambda^{-1}$ lie in the open
  interval $(\lambda^{-1}, \lambda)$. However, it was shown in the proof of
  \cite[Proposition 2.3]{StrennerSAF} that if $\lambda$ or $-\lambda$ is a root
  of $p(x)$, then~$\lambda^{-1}$ and~$-\lambda^{-1}$ cannot be roots of $p(x)$, hence
  mentioning the edge case $\pm \lambda^{-1}$ in the previous sentence is not
  necessary.

  If $p(x)$ was reciprocal, then $\lambda$ and $\lambda^{-1}$ or $-\lambda$ and
  $-\lambda^{-1}$ would have to be roots. If it was skew-reciprocal, then
  $\lambda$ and $-\lambda^{-1}$ or $-\lambda$ and $\lambda^{-1}$ would have to be
  roots. As we have just shown, these scenarios are impossible, because
  $\pm \lambda^{-1}$ is not a root of $p(x)$. This proves
  (\ref{item:other-roots-nonor}).

  The fact that $p(x)$ is reciprocal mod 2 was shown in \cite[Proposition
  4.2]{StrennerSAF}. This justifies (\ref{item:recip-mod-two}).

  Finally, notice that we have not guaranteed that $\lambda$ is a root of
  $p(x)$---we have only shown that either $\lambda$ or $-\lambda$ is a root. If
  it is $-\lambda$, then the polynomial $p(-x)$ or $-p(-x)$ satisfies all
  the required properties.
\end{proof}

We call a $2n \times 2n$ matrix $A$ \emph{anti-symplectic} if it the
corresponding linear transformation sends the standard symplectic form on
$\R^{2n}$ to its negative. Formally, this can be written as
\begin{displaymath}
  AJA^T = -J
\end{displaymath}
where $J =
\begin{pmatrix}
  0 & I \\
  -I & 0 \\
\end{pmatrix}$ and $I$ is the $n\times n$ identity matrix.

\begin{proposition}\label{prop:anti-simplectic-charpoly}
  The characteristic polynomial $p(x)$ of a $2n\times 2n$ anti-symplectic
  matrix is skew-reciprocal.
\end{proposition}
\begin{proof}
  Let $A$ be an $2n\times 2n$ anti-symplectic matrix. Since $A^2$ is
  symplectic, we have $\det(A^2) = 1$ and $\det(A) = \pm 1$. Since
  $\det(J) = 1$, we have
  \begin{align*}
    p(x) &= \det(A - xI) = \det(AJ - xJ) = \det(AJ + xAJA^T) \\
         &= \det(A)\det(J)\det(I+xA^T) =  \pm \det(I+xA)\\
         &= \pm x^{2n} \det\left(A + x^{-1}I\right) = \pm x^{2n} p(-x^{-1}),
  \end{align*}
  hence $p(x)$ is skew-reciprocal.
\end{proof}

The proof above is a straightforward modification of the standard proof of the
fact that the characteristic polynomials of symplectic matrices are reciprocal.

\begin{proposition}\label{prop:reversing-poly-properties}
  Let $\psi: S_g\to S_g$ be an orientation-reversing pseudo-Anosov map with
  transversely orientable invariant foliations. Then its stretch factor
  $\lambda$ is a root of a (not necessarily irreducible) polynomial
  $p(x) \in \Z[x]$ with the following properties:
  \begin{enumerate}
  \item\label{item:degree} $\deg(p) = 2g$.
  \item\label{item:const-coeff} $p(x)$ is monic and its constant coefficient is
    $(-1)^g$.
  \item\label{item:rev-poly} $p(x) = (-1)^g x^{2g}p(-x^{-1})$.
  \item\label{item:reciprocal-root} $p(-\lambda^{-1}) = 0$.
  \item\label{item:galois-conjugates-rev} The absolute values of the roots of
    $p(x)$ other than $\lambda$ and $-\lambda^{-1}$ lie in the open interval
    $(\lambda^{-1}, \lambda)$.
  \end{enumerate}
\end{proposition}
\begin{proof}
  Let $p(x)$ be the characteristic polynomial of
  $\psi_*: H_1(S_g) \to H_1(S_g)$. Clearly,~(\ref{item:degree}) holds. 
  Similarly to the proof of~\Cref{prop:nonor-poly-properties}, we 
  obtain (\ref{item:galois-conjugates-rev}) by a reduction to the known statement for 
  orientation-preserving pseudo-Anosov mapping classes. This time, 
  we directly obtain (\ref{item:galois-conjugates-rev}) by applying
  \cite[Theorem 5.3]{McMullen03a} to the square of $\psi$.

  An orientation-reversing homeomorphism sends the intersection form on
  $H_1(S_g)$ to its negative. \Cref{prop:anti-simplectic-charpoly} implies that
  $p(x) = \pm x^{2g}p(-x^{-1})$. To decide which sign is right, we only need to
  compute the sign of the constant coefficient of $p(x)$. If the constant
  coefficient $1$, then the sign is positive. If the constant coefficient is
  $-1$, then the sign is negative. To put this in another way, we have
  \begin{equation}\label{eq:recip}
    p(x) = p(0) x^{2g}p(-x^{-1}).
  \end{equation}

  For orientation-preserving homeomorphisms, the action on homology is
  symplectic, hence its determinant is $+1$. It follows that, for fixed $g$,
  the determinant is either $+1$ for all orientation-reversing homeomorphisms
  of $S_g$ or $-1$ for all orientation-reversing homeomorphisms of $S_g$. It is
  sufficient to check only one homeomorphism to decide which one. For example,
  consider the reflection $i:S_g \to S_g$ about the plane containing the curves
  $b_i$ on \Cref{fig:homology-basis}.

  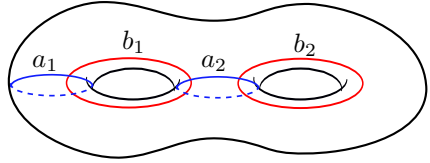
\begin{figure}[htb]
    \labellist
    \small\hair 2pt
    \pinlabel {$a_1$} [ ] at 14 37
    \pinlabel {$a_2$} [ ] at 77 37
    \pinlabel {$b_1$} [ ] at 47 45
    \pinlabel {$b_2$} [ ] at 111 44
    \endlabellist
    \centering
    \includegraphics[scale=1.0]{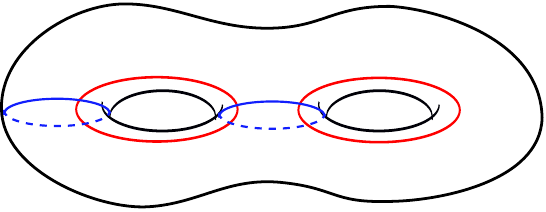}
    \caption{The standard homology basis for $S_2$.}
    \label{fig:homology-basis}
  \end{figure}

  The curves $\{a_1, \ldots, a_g, b_1, \ldots, b_g\}$ form a homology basis. We
  have $i(b_i) = b_i$ and~$i(a_i) = -a_i$ for all $i$. So the matrix of $A$ is
  a diagonal matrix whose diagonal contains $1$s and $-1$s, $g$ of each. Hence
  the determinant is $(-1)^g$ and we have shown~(\ref{item:const-coeff}).

  Item (\ref{item:rev-poly}) follows from equation (\ref{eq:recip}) and the
  fact $p(0) = (-1)^g$ we have just shown.

  Either $\lambda$ or $-\lambda$ is a root of $p(x)$. If it is $-\lambda$, we may replace
  $p(x)$ with $p(-x)$; the previously proven properties remain true. Finally,
  (\ref{item:reciprocal-root}) follows from (\ref{item:rev-poly}) by setting
  $x = \lambda$.
\end{proof}

We are now ready to prove \Cref{thm:Galois-nonor}. We emphasize that, unlike in
the previous propositions, in this theorem we are not assuming that the surface
is closed or that the pseudo-Anosov mapping class has a transversely orientable
invariant foliation.

\begin{proof}[Proof of \Cref{thm:Galois-nonor}]
  An irreducible polynomial $p(x) \in \Z[x]$ with a (complex) root~$\alpha$ on
  the unit circle is reciprocal. This is because then $\alpha^{-1}$ is also a
  root of $p(x)$, therefore $\alpha$ is a root of the polynomial
  $x^dp(x^{-1})$, where $d$ is the degree of $p(x)$. But the minimal polynomial
  is unique up constant factor, so $x^dp(x^{-1}) = \pm p(x)$. Hence~$p(x)$ is
  indeed reciprocal. So if a stretch factor $\lambda$ has a Galois conjugate on
  the unit circle, then the minimal polynomial of $\lambda$ is reciprocal and
  $\lambda^{-1}$ is also a root of the minimal polynomial.

  However, by \cite[Proposition 2.3]{StrennerSAF}, $\lambda$ and $\lambda^{-1}$
  are not Galois conjugates if $\lambda$ is a stretch factor of a pseudo-Anosov
  map (possibly with no orientable invariant foliations) on a nonorientable
  surface (possibly with punctures). This completes the proof in the case when
  the pseudo-Anosov map is supported on a nonorientable surface.

  We now prove the orientation-reversing case. If our surface is closed, then,
  by \Cref{prop:reversing-poly-properties}, $\lambda$ and $\lambda^{-1}$ are
  not Galois conjugates if $\lambda$ is a stretch factor of an
  orientation-reversing pseudo-Anosov map with orientable invariant foliations.
  If the foliations are not orientable, we can lift the map to the orientation
  double cover of the foliations to obtain an orientation-reversing
  pseudo-Anosov map with orientable invariant foliations and with the same
  stretch factor. Therefore $\lambda$ and $\lambda^{-1}$ are not Galois
  conjugates in this case, either.

  If our surface has punctures, then we can fill in the punctures after making
  the foliations orientable to obtain a pseudo-Anosov map with the same stretch
  factor on a closed surface, reducing to the closed case discussed in the
  previous paragraph. This completes the proof in the case when the
  pseudo-Anosov map is orientation-reversing.
\end{proof}

\section{Elimination of polynomials}
\label{sec:poly-elimination}

In this section we first prove bounds on the sum of $k$th powers of roots of a
polynomial when the absolute values of the roots are bounded by some $r > 1$.
These bounds are improved versions of Lemma~A.1.~of \cite{LanneauThiffeault11},
using the special properties of the polynomials in
\Cref{prop:nonor-poly-properties,prop:reversing-poly-properties}.

Then we describe how we use this lemma and
\Cref{prop:nonor-poly-properties,prop:reversing-poly-properties} in order to
systematically narrow down the set of
possible minimal polynomials of the minimal stretch factors.

\subsection{Power sum bounds}

We begin by proving two elementary lemmas.

\begin{lemma}\label{lemma:bounding-sum}
  Suppose $r > 1$ and $r^{-1} \le a_1, \ldots, a_d \le r$ are positive real numbers such that $a_1\cdots a_d = 1$. Then
  \begin{displaymath}
    \sum_{i=1}^d a_i \le
    \begin{cases}
      n(r + r^{-1}) & \mbox{if $d=2n$ is even}\\
      n(r + r^{-1}) + 1 & \mbox{if $d=2n+1$ is odd.}
    \end{cases}
  \end{displaymath}
\end{lemma}
\begin{proof}
  The function $x \mapsto x + x^{-1}$ is increasing on the interval $x \ge 1$.
  So if there are~$i\ne j$ so that $r^{-1} < a_i, a_j < r$, then we can
  increase the sum by moving $a_i$ and~$a_j$ away from each other by keeping
  their product unchanged, until at least one of them is $r^{-1}$ or $r$. After
  every such operation, the number of $a_i$ that are equal to~$r^{-1}$ or~$r$
  increases. So eventually we get to a point where at most one $a_i$ is not
  $r^{-1}$ or~$r$. When~$d=2n$, no such $a_i$ can exist, and exactly half of
  the $a_i$ equal $r$, the other half~$r^{-1}$, otherwise their product would not
  be 1. When $d=2n+1$, exactly one such$a_i$ exists, it equals 1 and exactly
  half of the remaining $a_i$ equal $r$, the other half $r^{-1}$.
\end{proof}

\begin{lemma}\label{lemma:bounding-diff}
  Suppose $r > 1$ and $a_1, \ldots, a_d$ are positive real numbers such
  that $r^{-1} \le a_1 \le \ldots \le a_d \le r$ and $a_1\cdots a_d = 1$ and
  $a_1 \ge a_d^{-1}$. Then
  \begin{displaymath}
    a_d - \sum_{i=1}^{d-1} a_i \ge
    \begin{cases}
      \min \{2-2n, -(n-2)r - nr^{-1}\} & \mbox{if $d=2n$ is odd.}\\
      \min \{1-2n, -(n-2)r -1- nr^{-1}\} & \mbox{if $d=2n+1$ is odd.}
    \end{cases}
  \end{displaymath}
  Moreover, the inequalities are strict if $a_1 > a_d^{-1}$.
\end{lemma}
\begin{proof}
  Similarly to the proof of \Cref{lemma:bounding-sum}, our approach is to change
  the numbers $a_1, \ldots, a_{d-1}$ to increase $\sum_{i=1}^{d-1} a_i$ as much
  as possible while keeping the hypotheses true. Whenever there are $i\ne j$
  such that $a_d^{-1} < a_i, a_j, < a_d$, we push $a_i$ and $a_j$ apart until
  at least one of them equals $a_d^{-1}$ or $a_d$. The end result is the same
  as before, so we have
  \begin{displaymath}
    a_d - \sum_{i=1}^{d-1} a_i \ge a_d - ((n-1)a_d + n a_d^{-1}) = -(n-2)a_d - na_d^{-1}.
  \end{displaymath}
  when $d = 2n$ and
  \begin{displaymath}
    a_d - \sum_{i=1}^{d-1} a_i \ge a_d - ((n-1)a_d + 1 + n a_d^{-1}) = -(n-2)a_d -1- na_d^{-1}
  \end{displaymath}
  when $d = 2n+1$. Since the function $x \mapsto -(n-2)x - nx^{-1}$ is concave,
  its minimum on the interval $[1,r]$ is taken at one of the endpoints.

  The inequalities in the case $a_1 > a_d^{-1}$ are strict, since in the
  optimal distribution there has to be an $a_i$ that takes the value $a_d^{-1}$.
\end{proof}

Now we apply \Cref{lemma:bounding-sum,lemma:bounding-diff} for roots of
polynomials.

\begin{cor}\label{cor:power-sum-improved-bounds}
  Suppose $P(x)$ is a monic polynomial of degree $d$ with constant coefficient
  $\pm1$. Let $z_1, \ldots, z_d$ be the roots of $P(x)$ and let
  \begin{displaymath}
    p_k = z_1^k + \cdots + z_d^k
  \end{displaymath}
  the $k$th power sum of the roots.

  Suppose there is a root $\lambda > 1$ such that all the other roots
  have absolute values in the interval $[\lambda^{-1},\lambda]$. For any
  $r>\lambda$, we have
  \begin{displaymath}
    \min \{2-2n, -(n-2)r^k - nr^{-k}\} \le p_k \le n(r^k + r^{-k})
  \end{displaymath}
  if $d = 2n$ is even and
  \begin{displaymath}
    \min \{1-2n, -(n-2)r^k -1-nr^{-k}\} \le p_k \le n(r^k + r^{-k}) + 1
  \end{displaymath}
  if $d = 2n+1$ is odd.

  Moreover, strict inequality holds in the lower bound when no eigenvalue
  equals $\lambda^{-1}$.
\end{cor}
\begin{proof}
  Let $z_1,\ldots,z_d$ be the roots of $P(x)$ and let $a_i = |z_i|^k$ for every
  $i$. Note that $a_1\cdots a_d = 1$ and $r^{-k} \le a_1, \ldots, a_d \le r^k$.
  Assuming $a_1 \le \cdots \le a_d$, we have $a_d = \lambda^k$. Since
  \begin{displaymath}
    a_d - \sum_{i=1}^{d-1} a_i \le p_k = z_1^k + \cdots + z_{d-1}^k + \lambda^k \le \sum_{i=1}^d a_i,
  \end{displaymath}
  the bounds follow from \Cref{lemma:bounding-sum} and \Cref{lemma:bounding-diff}
\end{proof}

\subsection{Newton's formulas}

In this section, we recall Newton's formulas that relate power sums of the
roots to the coefficients of the polynomial.

We will use the notation
\begin{displaymath}
  P(x) = x^d - c_1x^{d-1} - \cdots -c_{d-1}x \pm 1
\end{displaymath}
for the coefficients of monic polynomials of degree $d$. As in the statement of
\Cref{cor:power-sum-improved-bounds}, we denote by $p_k$ the $k$th power sum
of the roots of $P(x)$.

Newton's formulas relating power sums and symmetric polynomials can be stated
either as
\begin{equation}
  \label{eq:expressing-power-sum}
  p_k = - c_1p_{k-1} - c_2p_{k-2}- \cdots -  c_{k-1}p_1 - kc_k
\end{equation}
or as
\begin{equation}
  \label{eq:expressing-coeff}
  c_k = \frac{-c_1p_{k-1} - c_2p_{k-2}- \cdots -  c_{k-1}p_1 - p_k}k
\end{equation}
for all $1\le k \le d-1$.

As Lanneau and Thiffeault point out in Section~A.1 of \cite{LanneauThiffeault11},
is it more computationally efficient to bound the power sums $p_i$ and using
Newton's formulas to compute the coefficients $c_k$ from the $p_i$ than to
bound the coefficients directly. This is because many scenarios get ruled
out just because the numerator in equation (\ref{eq:expressing-coeff}) is not divisible
by~$k$.

\subsection{The polynomial elimination algorithm}
\label{sec:algorithm}

We give a lower bound on the minimal stretch factor $\delta^+(N_g)$ by a
systematic elimination of polynomials. We describe this process below. In order
to illustrate the effect of each step in the algorithm, we give the number of
candidate polynomials left after each step when $g=12$ (when the degree is 11).

\begin{algorithm}\label{alg}
  Let $d \ge 4$, $g = d-1$ and $r > 1$ such that $\delta^+(N_g) < r$. Perform
  the following steps in order to obtain a small set of polynomials of degree
  $g$ that include one polynomial whose root is $\delta^+(N_g)$:
  \begin{enumerate}
  \item\label{item:power-sum-bounds} Compute the possible values of
    $p_1, \ldots, p_{d-1}$ using the bounds given by
    \Cref{cor:power-sum-improved-bounds}. For $d=11$, the total number of
    combinations is
    $20 \cdot 20\cdot 21\cdot 23\cdot 24\cdot 27\cdot 30\cdot 34\cdot 38\cdot
    43 = 10,641,541,131,648,000$.
  \item\label{item:compute-coeffs} Compute the coefficients
    $c_1, \ldots, c_{d-1}$ using \Cref{eq:expressing-coeff}, keeping only the
    cases when all $c_i$ are integers. 57,643,952 cases remain.
  \item\label{item:mod-two} Discard all cases where the polynomial is not
    reciprocal mod 2. 1,808,922 cases remain.
  \item\label{item:add-constant-coeff} Try $\pm 1$ for the constant
    coefficient. We now doubled the number of cases to 3,617,844.
  \item\label{item:backward-powers} Consider the reciprocal polynomial
    $P^*(x) = \pm x^dP(x^{-1})$ (with the sign chosen so that the polynomial is
    monic), and use \Cref{eq:expressing-power-sum} to compute the power sums~
    $p^*_1, \ldots, p^*_{d-1}$ of this polynomial from the reversed sequence~
    $\pm c_{d-1}, \ldots, \pm c_1$ of coefficients, where the signs here depend
    on the sign chosen in the previous step. Discard the cases that do not
    satisfy the bounds of \Cref{cor:power-sum-improved-bounds}. 5075 cases
    remain.
  \item\label{item:newton} Test the remaining polynomials by Newton's method
    for finding roots. Start with the upper bound for the Perron root. Since
    the polynomial is increasing and convex in $[\lambda, \infty)$, we should
    get a decreasing sequence of $x$-values larger than~$\lambda$. Discard the
    cases when this fails. Stop when two consecutive $x$-values are very close
    to each other. 421 cases remain.
  \item Discard the polynomials where the largest eigenvalue in absolute value
    is not real. 86 cases remain.
  \item Discard the cases where the multiplicity of the largest eigenvalue is
    larger than 1. 54 cases remain.
  \item Discard the cases when there is a root with absolute value less than or
    equal to~$\lambda^{-1}$. 33 cases remain.
  \item Discard the cases where the largest eigenvalue is larger than our upper
    bound. 1 case remains.
  \end{enumerate}
\end{algorithm}

In practice, the first three steps are implemented in a more sophisticated way.
Our computers cannot handle as many as $10,641,541,131,648,000$ cases, so the
implementation does not actually consider all those combinations. It first
chooses a value for~$p_1$ and sets $c_1 = -p_1$ by (\ref{eq:expressing-coeff}).
Then it has to choose a value for $p_2$ of the same parity as~$p_1$, since
$c_2 = -\frac{c_1p_1-p_2}{2}$. Similarly, the value chosen for $p_3$ is then
determined mod 3, therefore a huge number of combinations for the $p_i$ are
never considered. Also, once more than half of the $c_i$ are computed, we
obtain additional constraints on the $p_i$, since our polynomial has to be
reciprocal mod 2. Since these divisibility checks are done early, and not after
the whole polynomial is constructed, a huge number of cases gets eliminated
early.

The idea of bounding the coefficients using power sums as in steps
(\ref{item:power-sum-bounds}) and (\ref{item:compute-coeffs}) instead of using
symmetric polynomials to express the coefficients in terms of the roots is due
to Lanneau and Thiffeault \cite{LanneauThiffeault11}. As the numbers suggest,
these are the steps that are responsible for bringing the size of the set of
possible polynomials down from an astronomical size to one that is approachable
by computers.

Step (\ref{item:mod-two}) is special to nonorientable surfaces and is also
crucial. Without this step, not only would the searching process be much
slower, but there are quite a few polynomials that pass all the other tests but
this is the only step that eliminates them. Perhaps this step is the main
reason for why we do not need Lefschetz number tests unlike Lanneau and
Thiffeault in the orientable case.

Step (\ref{item:backward-powers}) is also special to nonorientable surfaces,
since in the orientable case the polynomials are reciprocal, so the reciprocal
polynomial does not contain any additional information. This was one of the
last tests we added, and this reduced the running time of the algorithm for the
$d=11$ case from several hours to a few minutes. The reason this works so
effectively is that in most of the coefficient sequences at this point, the
last few coefficients ($c_{d-1}$, $c_{d-2}$, etc.) are much bigger than the
required bounds.

Step (\ref{item:newton}) is another computationally inexpensive test that
quickly eliminates a large fraction of the polynomials. The idea of this test
is also due to Lanneau and Thiffeault.

The most computationally expensive part is computing the roots. We only compute
the roots after Step (\ref{item:newton}), only in 421 cases. So in terms of
total time, actually steps (\ref{item:power-sum-bounds})--(\ref{item:newton}) take
more than 99\% of the running time.

We use a very similar algorithm in order to give a lower bound for the minimal
stretch factor~$\delta^+_{rev}(S_g)$ among orientation-reversing pseudo-Anosov
maps. The difference is that we use the properties from
Proposition~\ref{prop:reversing-poly-properties} instead of the ones from
Proposition~\ref{prop:nonor-poly-properties}. Our implementation of these
algorithms can be found at
\href{https://github.com/b5strbal/polynomial-filtering}{\texttt{https://github.com/b5strbal/polynomial-filtering}}.

\subsection{Minimal stretch factors}
\label{sec:final-section}

We are now ready to single out the minimal stretch factor $\delta^+(N_g)$ among
pseudo-Anosov homeomorphisms with an orientable invariant foliation for certain
nonorientable closed surfaces~$N_g$. \Cref{theorem:stretch_factors_nonor} is a
direct consequence of \Cref{cor:construction-nonor} and
\Cref{prop:elimination-nonor} below.

\begin{proposition}\label{prop:elimination-nonor}
  Let $g$ and $r$ be as in one of the rows in the table below. Let $f$ be a
  pseudo-Anosov mapping class with an orientable invariant foliation on $N_g$
  whose stretch factor $\lambda$ is smaller than $r$. Then $\lambda$ must be a
  root of the polynomial shown in the table.

  In the cases where no polynomial is given, we have indicated to how many
  polynomials we were able to restrict the list of candidate polynomials.
  \begin{center}
    \begin{tabular}{c|c|c|c}
      $g$ & $r$ & Polynomial candidates & largest root \\
      \hline
      4 & 1.84 & $x^3 - x^2 - x - 1$ & 1.83929 \\
      5 & 1.52 & $x^4 - x^3 - x^2 + x - 1$ & 1.51288\\
      6 & 1.43 & $x^5 - x^3 - x^2 - 1$ & 1.42911\\
      7 & 1.422 & $x^6 - x^5 - x^3 + x - 1$ & 1.42198\\
      8 & 1.2885 & $x^7 - x^4 - x^3 - 1$ & 1.28845\\
      9 & 1.3568 & \emph{18 candidates} & \\
      10 & 1.2173 & $x^9 - x^5 - x^4 - 1$ & 1.21728\\
      11 & 1.22262 & \emph{5 candidates}& \\
      12 & 1.1743 & $x^{11} - x^6 - x^5 - 1$ & 1.17429 \\
      13 & 1.2764 & \emph{288 candidates} & \\
      14 & 1.14552 & $x^{13} - x^7 - x^6 - 1$ & 1.14551\\
      15 & 1.1875 & \emph{84 candidates} & \\
      16 & 1.1249 & $x^{15} - x^8 - x^7 - 1$ & 1.12488 \\
      17 & 1.1426 & \emph{16 candidates} & \\
      18 & 1.10939 & $x^{17} - x^9 - x^8 - 1$ & 1.10938\\
      20 & 1.09731 & $x^{19} - x^{10} - x^9 - 1$ & 1.09730 \\
    \end{tabular}
  \end{center}
\end{proposition}
\begin{proof}
  The proof consists of running \Cref{alg} and is computer-assisted. However,
  we will prove the proposition by hand in genus 4. We follow the first four
  steps explicitly, then (since only a handful of polynomials remain) we finish
  the proof with an ad hoc but simple argument.

  Step (\ref{item:power-sum-bounds}): we have
  \begin{displaymath}
    -1 = \min\{-1,1.84 - 1 - 1/1.84\} < p_1 < 1.84 + 1 + 1/1.84 \approx 3.38,
  \end{displaymath}
  therefore the possible values for $p_1$ are 0, 1, 2 and 3. We have
  \begin{displaymath}
    -1 = \min\{-1,1.84^2 - 1 - 1/1.84^2\}< p_2 < 1.84^2 + 1 + 1/1.84^2 \approx 4.68,
  \end{displaymath}
  so the possible values for $p_2$ are 0, 1, 2, 3 and 4.

  Step (\ref{item:compute-coeffs}): By (\ref{eq:expressing-coeff}), we have
  $c_1 = -p_1$ and $c_2 = \frac{p_1^2-p_2}2$, therefore $p_1$ and $p_2$ have
  the same parity. Hence the possible pairs are
  $(0,0)$, $(0,2)$, $(0,4)$, $(1,1)$, $(1,3)$, $(2,0)$, $(2,2)$, $(2,4)$, $(3,
  1)$ and $(3,3)$.

  Step (\ref{item:mod-two}): The pair $(p_1, \frac{p_1^2-p_2}2)$ also has the
  same parity, since our polynomial is reciprocal mod 2. That leaves the
  choices $(0,0)$, $(0,4)$, $(1,3)$, $(2,0)$, $(2,4)$, $(3,3)$ for $(p_1,p_2)$.

  Step (\ref{item:add-constant-coeff}): We construct the list of possible polynomials.
  \begin{enumerate}
  \item $x^3 \pm 1$
  \item $x^3 -2x \pm 1$
  \item $x^3 - x^2-x\pm 1$
  \item $x^3-2x^2+2x\pm 1$
  \item $x^3-2x^2 \pm 1$
  \item $x^3-3x^2+3x\pm 1$
  \end{enumerate}
  The polynomial has to be irreducible, since the degree of a stretch factor on
  a nonorientable surface is at least three \cite[Proposition
  8.7.]{StrennerDegrees}. The polynomials where neither 1 nor $-1$ are roots are
  $x^3-x^2-x-1$, $x^3-2x^2+2x+1$, $x^3-2x^2-1$ and $x^3-3x^2+3x+1$. The second
  and fourth polynomial do not have a positive real root, and the third
  polynomial has a root that is approximately 2.2. That leaves us with
  $x^3-x^2-x-1$.
\end{proof}

We have stopped at genus 20 because of computational difficulties. The genus 18
case took about half a day to run on a single computer. In the genus 20 case
the algorithm took about a day to complete when run parallel on 30
computers. We estimate that the genus 22 case would need to run for a few
months on the same cluster of computers.

In the odd genus cases, the issue is not the running time, but the fact that
our tests are not good enough to eliminate all polynomials that should be
eliminated. In the hope of dealing with more odd genus cases, we have also
implemented the Lefschetz number tests used by Lanneau and Thiffeault
\cite[Section 2.3]{LanneauThiffeault11}. These tests help eliminate a large
percentage of the remaining polynomials, but, unfortunately, not all. \Cref{tab:noLefschetz} 
below shows the polynomials that we could not eliminate in the genus
9, 11 and 13 cases, in addition to the polynomials that we have constructed in
\Cref{cor:construction-nonor}.

\begin{table}[ht]
\begin{center}
  \begin{tabular}[ht]{c|c|c|c}
    $g$ & Polynomial & Stratum &  Stretch factor \\
    \hline
    9 & $(x^7 - x^4 - x^3 - 1)(x - 1)$ & $(4^7)$ & 1.28845 \\
    & $(x^7 - x^5 - x^2 - 1)(x - 1)$ & $(4^7)$ & 1.30740 \\
    \hline
    11 & $(x^9 - x^5 - x^4 - 1)(x - 1)$ & $(4^9)$ & 1.21728 \\
    \hline
    13 & 18 polynomials & & 
  \end{tabular}
  \caption{The polynomials and possible strata in genus 9, 11 that we cannot
    rule out using \Cref{alg} and Lefschetz arguments.
    The notation $a^b$ means an orbit of length $b$ consisting of $a$-pronged
    singularities. }
    \label{tab:noLefschetz}
\end{center}
\end{table}

Most of the polynomials that we are not able to eliminate are products of
polynomials that appear in some lower genus and cyclotomic polynomials. We
think that these polynomials should be possible to eliminate, but we do not
know how. In particular, we think that in the genus 9 and 11 cases all three
remaining polynomials could be eliminated, and we conjecture that the examples
constructed in \Cref{cor:construction-nonor} are the minimal stretch factor
examples (cf.~\Cref{conj:nine_eleven}).

Similarly, in the orientation-reversing case,
\Cref{theorem:stretch_factors_rev} follows directly from
\Cref{cor:construction-reversing} and
\Cref{prop:elimination-reversing} below.

\begin{proposition}\label{prop:elimination-reversing}
  Let $g$ and $r$ be as in one of the rows in the table below. Let $f$ be an
  orientation-reversing pseudo-Anosov mapping class with orientable invariant
  foliations on $S_g$ whose stretch factor $\lambda$ is smaller than $r$. Then
  $\lambda$ must be a root of the polynomial shown in the table.
  \begin{center}
    \begin{tabular}{c|c|c|c}
      $g$ & $r$ & Polynomial candidates & largest root \\
      \hline
      1 & 1.62 & $x^2-x-1$ & 1.61803 \\
      2 & 1.62 & $x^2-x-1$ & 1.61803 \\
      3 & 1.253 & $x^8-x^5-x^3-1$ & 1.25207\\
      4 & 1.253 & $x^8-x^5-x^3-1$ & 1.25207\\
      5 & 1.16 & $x^{12}-x^7-x^5-1$ & 1.15973\\
      6 & 1.16 & $x^{12}-x^7-x^5-1$ & 1.15973\\
      7 & 1.1171 & $x^{16}-x^9-x^7-1$ & 1.11707\\
      8 & 1.1171 & $x^{16}-x^9-x^7-1$ & 1.11707\\
      9 & 1.0925 & $x^{20}-x^{11}-x^9-1$ & 1.09244\\
      10 & 1.0925 & $x^{20}-x^{11}-x^9-1$ & 1.09244\\
      11 & 1.0764 & $x^{24}-x^{13}-x^{11}-1$ & 1.07638\\
    \end{tabular}
  \end{center}
\end{proposition}
\begin{proof}
  Analogously to the proof of Proposition~\ref{prop:elimination-nonor},
  the proof of this statement is also computer-assisted. The algorithm used is
  a slight modification of \Cref{alg} as mentioned at the end of
  \Cref{sec:algorithm}.

  The polynomials in the table for $g\ge 3$ are not irreducible: they are
  products of $x^2+1$ and an irreducible factor. When $g\ge 3$ is odd, the
  polynomial we get as a result of the elimination process is this irreducible
  factor. When $g\ge 4$ is even, then the only polynomial left is the product
  of that irreducible factor and $x^2-1$. In either case, the stretch factor
  has to be a root of the irreducible factor, therefore a root of the
  polynomials in the table. (The reasons for why we have not listed the
  irreducible factors in the table is that they have many more terms and they do
  not show such a clear pattern as the polynomials in the table. Moreover, the
  polynomials in the table appear also in \Cref{cor:construction-reversing}.)

  Similarly, in genus 2, the polynomial remaining after the elimination process is
  $(x^2-x-1)(x^2-1)$, so the stretch factor would have to be a root of~$x^2-x-1$.
\end{proof}

By using the Lefschetz number arguments of Lanneau and Thiffeault, we think it
is possible to show that in genus 2 the only remaining polynomial, $(x^2-x-1)(x^2-1)$,
cannot actually be the characteristic polynomial of the action on the first homology for
an orientation-reversing pseudo-Anosov map with orientable invariant
foliations. This would imply that $\delta^+_{rev}(S_2) > \delta^+_{rev}(S_1)$.
Since \Cref{prop:elimination-reversing} shows a very clear pattern, we
conjecture that the stretch factor candidates in
\Cref{prop:elimination-reversing} cannot be realized for any even genus,
leading to \Cref{conj:zigzag}.

\bibliographystyle{alpha}
\bibliography{mybibfile}

\begin{thebibliography}{GMM09}

\bibitem[AD10]{AaberDunfield10}
John~W. Aaber and Nathan Dunfield.
\newblock Closed surface bundles of least volume.
\newblock {\em Algebr. Geom. Topol.}, 10(4):2315--2342, 2010.

\bibitem[Ada87]{Adams87}
Colin~C. Adams.
\newblock The noncompact hyperbolic {$3$}-manifold of minimal volume.
\newblock {\em Proc. Amer. Math. Soc.}, 100(4):601--606, 1987.

\bibitem[Ago02]{Agol02}
Ian Agol.
\newblock Volume change under drilling.
\newblock {\em Geom. Topol.}, 6:905--916, 2002.

\bibitem[AST07]{AgolStormThurston07}
Ian Agol, Peter~A. Storm, and William~P. Thurston.
\newblock Lower bounds on volumes of hyperbolic {H}aken 3-manifolds.
\newblock {\em J. Amer. Math. Soc.}, 20(4):1053--1077, 2007.
\newblock With an appendix by Nathan Dunfield.

\bibitem[AY81]{ArnouxYoccoz81}
Pierre Arnoux and Jean-Christophe Yoccoz.
\newblock Construction de diff\'eomorphismes pseudo-{A}nosov.
\newblock {\em C. R. Acad. Sci. Paris S\'er. I Math.}, 292(1):75--78, 1981.

\bibitem[Bau92]{Bauer92}
Max Bauer.
\newblock An upper bound for the least dilatation.
\newblock {\em Trans. Amer. Math. Soc.}, 330(1):361--370, 1992.

\bibitem[Bre51]{Breusch51}
Robert Breusch.
\newblock On the distribution of the roots of a polynomial with integral
  coefficients.
\newblock {\em Proc. Amer. Math. Soc.}, 2:939--941, 1951.

\bibitem[CH08]{ChoHam08}
Jin-Hwan Cho and Ji-Young Ham.
\newblock The minimal dilatation of a genus-two surface.
\newblock {\em Experiment. Math.}, 17(3):257--267, 2008.

\bibitem[GMM09]{GabaiMeyerhoffMilley09}
David Gabai, Robert Meyerhoff, and Peter Milley.
\newblock Minimum volume cusped hyperbolic three-manifolds.
\newblock {\em J. Amer. Math. Soc.}, 22(4):1157--1215, 2009.

\bibitem[Hir10]{Hironaka10}
Eriko Hironaka.
\newblock Small dilatation mapping classes coming from the simplest hyperbolic
  braid.
\newblock {\em Algebr. Geom. Topol.}, 10(4):2041--2060, 2010.

\bibitem[HK06]{HironakaKin06}
Eriko Hironaka and Eiko Kin.
\newblock A family of pseudo-{A}nosov braids with small dilatation.
\newblock {\em Algebr. Geom. Topol.}, 6:699--738 (electronic), 2006.

\bibitem[Iva88]{Ivanov88}
N.~V. Ivanov.
\newblock Coefficients of expansion of pseudo-{A}nosov homeomorphisms.
\newblock {\em Zap. Nauchn. Sem. Leningrad. Otdel. Mat. Inst. Steklov. (LOMI)},
  167(Issled. Topol. 6):111--116, 191, 1988.

\bibitem[KKT09]{KinKojimaTakasawa09}
E.~Kin, S.~Kojima, and M.~Takasawa.
\newblock Entropy versus volume for pseudo-{A}nosovs.
\newblock {\em Experiment. Math.}, 18(4):397--407, 2009.

\bibitem[KM18]{KojimaMcShane18}
Sadayoshi Kojima and Greg McShane.
\newblock Normalized entropy versus volume for pseudo-{A}nosovs.
\newblock {\em Geom. Topol.}, 22(4):2403--2426, 2018.

\bibitem[KT13]{KinTakasawa13}
Eiko Kin and Mitsuhiko Takasawa.
\newblock Pseudo-{A}nosovs on closed surfaces having small entropy and the
  {W}hitehead sister link exterior.
\newblock {\em J. Math. Soc. Japan}, 65(2):411--446, 2013.

\bibitem[Lei04]{Leininger04}
Christopher~J. Leininger.
\newblock On groups generated by two positive multi-twists: {T}eichm\"uller
  curves and {L}ehmer's number.
\newblock {\em Geom. Topol.}, 8:1301--1359 (electronic), 2004.

\bibitem[Lie18]{LiechtiSZ}
Livio Liechti.
\newblock On the arithmetic and the geometry of skew-reciprocal polynomials.
\newblock {\em preprint}, 2018, to appear in Proc.\ Amer.\ Math.\ Soc.

\bibitem[LS18a]{LiechtiStrennerAY}
Livio Liechti and Bal\'azs Strenner.
\newblock The {A}rnoux--{Y}occoz mapping classes via {P}enner's construction.
\newblock {\em preprint}, 2018, to appear in Bull.\ Soc.\ Math.\ France.

\bibitem[LS18b]{LiechtiStrennerPenner}
Livio Liechti and Bal\'azs Strenner.
\newblock Minimal {P}enner dilatations on nonorientable surfaces.
\newblock {\em preprint}, 2018, to appear in J.\ Topol.\ Anal.

\bibitem[LT11a]{LanneauThiffeault11a}
Erwan Lanneau and Jean-Luc Thiffeault.
\newblock On the minimum dilatation of braids on punctured discs.
\newblock {\em Geom. Dedicata}, 152:165--182, 2011.
\newblock Supplementary material available online.

\bibitem[LT11b]{LanneauThiffeault11}
Erwan Lanneau and Jean-Luc Thiffeault.
\newblock On the minimum dilatation of pseudo-{A}nosov homeromorphisms on
  surfaces of small genus.
\newblock {\em Ann. Inst. Fourier (Grenoble)}, 61(1):105--144, 2011.

\bibitem[McM00]{McMullen00}
Curtis~T. McMullen.
\newblock Polynomial invariants for fibered 3-manifolds and {T}eichm\"uller
  geodesics for foliations.
\newblock {\em Ann. Sci. \'Ecole Norm. Sup. (4)}, 33(4):519--560, 2000.

\bibitem[McM03]{McMullen03a}
Curtis~T. McMullen.
\newblock Billiards and {T}eichm{\"u}ller curves on {H}ilbert modular surfaces.
\newblock {\em J. Amer. Math. Soc.}, 16(4):857--885 (electronic), 2003.

\bibitem[Mil09]{Milley09}
Peter Milley.
\newblock Minimum volume hyperbolic 3-manifolds.
\newblock {\em J. Topol.}, 2(1):181--192, 2009.

\bibitem[Min06]{Minakawa06}
Hiroyuki Minakawa.
\newblock Examples of pseudo-{A}nosov homeomorphisms with small dilatations.
\newblock {\em J. Math. Sci. Univ. Tokyo}, 13(2):95--111, 2006.

\bibitem[Pen88]{Penner88}
Robert~C. Penner.
\newblock A construction of pseudo-{A}nosov homeomorphisms.
\newblock {\em Trans. Amer. Math. Soc.}, 310(1):179--197, 1988.

\bibitem[Pen91]{Penner91}
Robert~C. Penner.
\newblock Bounds on least dilatations.
\newblock {\em Proc. Amer. Math. Soc.}, 113(2):443--450, 1991.

\bibitem[SS15]{ShinStrenner15}
Hyunshik Shin and Bal{{\'a}}zs Strenner.
\newblock Pseudo-{A}nosov mapping classes not arising from {P}enner's
  construction.
\newblock {\em Geom. Topol.}, 19(6):3645--3656, 2015.

\bibitem[Str17]{StrennerDegrees}
Bal\'azs Strenner.
\newblock Algebraic degrees of pseudo-{A}nosov stretch factors.
\newblock {\em Geom. Funct. Anal.}, 27(6):1497--1539, 2017.

\bibitem[Str18]{StrennerSAF}
Bal\'azs Strenner.
\newblock Lifts of pseudo-{A}nosov homeomorphisms of nonorientable surfaces
  have vanishing {SAF} invariant.
\newblock {\em Math. Res. Lett.}, 25(2):677--685, 2018.

\bibitem[SZ65]{SchinzelZassenhaus65}
A.~Schinzel and H.~Zassenhaus.
\newblock A refinement of two theorems of {K}ronecker.
\newblock {\em Michigan Math. J.}, 12:81--85, 1965.

\bibitem[Thu88]{Thurston88}
William~P. Thurston.
\newblock On the geometry and dynamics of diffeomorphisms of surfaces.
\newblock {\em Bull. Amer. Math. Soc. (N.S.)}, 19(2):417--431, 1988.

\bibitem[Tsa09]{Tsai09}
Chia-Yen Tsai.
\newblock The asymptotic behavior of least pseudo-{A}nosov dilatations.
\newblock {\em Geom. Topol.}, 13(4):2253--2278, 2009.

\bibitem[Val12]{Valdivia12}
Aaron~D. Valdivia.
\newblock Sequences of pseudo-{A}nosov mapping classes and their asymptotic
  behavior.
\newblock {\em New York J. Math.}, 18:609--620, 2012.

\bibitem[Yaz18]{Yazdi18}
Mehdi Yazdi.
\newblock Pseudo-{A}nosov maps with small stretch factors.
\newblock {\em Preprint}, 2018.

\bibitem[Zhi95]{Zhirov95}
A.~Yu. Zhirov.
\newblock On the minimum dilation of pseudo-{A}nosov diffeomorphisms of a
  double torus.
\newblock {\em Uspekhi Mat. Nauk}, 50(1(301)):197--198, 1995.

\end{thebibliography}

\end{document}